\documentclass[a4paper,12pt]{article}

\usepackage{amsmath,amssymb,amsthm, mathrsfs}
\usepackage{latexsym}
\usepackage{graphics}
\usepackage{hyperref}
\usepackage{graphicx,psfrag}
\usepackage[bf, small]{titlesec}
\usepackage{authblk} %%% new line between authors
\usepackage{soul}
\usepackage{kotex}
\usepackage{lineno}
\theoremstyle{plain}
\newtheorem{Thm}{Theorem}[section]
\newtheorem{Lem}[Thm]{Lemma}
\newtheorem{Prop}[Thm]{Proposition}
\newtheorem{Cor}[Thm]{Corollary}

\theoremstyle{definition}

\usepackage{standalone}
\usepackage{pgfpages,tikz,pgfkeys,pgfplots}
\usetikzlibrary{arrows,positioning,matrix,fit,backgrounds,shapes,shapes.geometric, intersections}
\usetikzlibrary{shapes.callouts,decorations.pathmorphing} %%% 팝업용
\usetikzlibrary{shadows} % shadow
\usetikzlibrary{arrows.meta}
\usepackage{calc} % calculate angles to evenly space vertices in circular arrangements.
\usepackage{float}
\usetikzlibrary{decorations.markings} %%put arrowheads in the middle of directed edges
\tikzstyle{vertex}=[circle, draw, inner sep=0pt, minimum size=6pt] % style
\newcommand{\vertex}{\node[vertex]}
\usetikzlibrary{arrows,matrix} %%% Hesse Diagram

\title{ The triangle-free graphs that are competition graphs of multipartite tournaments}

\author[1]{Myungho Choi
}
\author[1]{Minki Kwak
}
\author[1]{Suh-Ryung Kim
}

\affil[1]{Department of Mathematics Education,
Seoul National University, Seoul 08826, Republic of Korea}

\begin{document}
\maketitle
\begin{abstract}
In this paper, we discover all the triangle-free graphs that are competition graphs of multipartite tournaments.
%
%we discover that
%all the triangle-free graphs that are competition graphs of multipartite tournaments and
%the competition graph of a $k$-partite tournament cannot be triangle-free if $k\geq 6$.
%
% show that there is no triangle-free graphs 
%
%we show that a connected triangle-free graph is the competition graph of a $k$-partite tournament if and only if $k \in \{3,4,5\}$, and a disconnected triangle-free graph is the competition graph of a $k$-partite tournament if and only if $k \in \{2,3,4\}$.
%Then we list all the triangle-free graphs in each case.
\end{abstract}

\section{Introduction}

Given a digraph $D$,
$N^+_D(x)$ and $N^-_D(x)$ denote the sets of out-neighbors and in-neighbors, respectively, of a vertex $x$ in $D$. The nonnegative integers $|N^+_D(x)|$ and $|N^-_D(x)|$ are called the {\it outdegree} and the {\it indegree}, respectively, of $x$ and denoted by $d_D^+(x)$ and $d_D^-(x)$, respectively.
%문장 나눠서 정정함. 주어가 다른데 and로 연결되어 있었음.
When no confusion is likely, we omit $D$ in $N^+_{D}(x)$, $N^-_{D}(x)$, $d_D^+(x)$, and $d_D^-(x)$, to just write $N^+(x)$, $N^-(x)$, $d^+(x)$, and $d^-(x)$, respectively.

The \emph{competition graph} $C(D)$ of a digraph $D$
is the (simple undirected) graph $G$ defined by
$V(G)=V(D)$ and $E(G)=\{uv \mid u,v \in V(D), u \neq v,
N_D^+(u) \cap N_D^+(v) \neq \emptyset \}$.
Competition graphs arose in
connection with an application in ecology (see \cite{Cohen})
and also have applications in coding,
radio transmission, and modeling of complex economic systems.
Early literature of the study on competition graphs is summarized
in the survey papers by Kim~\cite{kim1993competition} and Lundgren~\cite{lundgren1989food}.

For a digraph $D$,
the \emph{underlying graph} of $D$
is the graph $G$
such that $V(G)=V(D)$ and $E(G)=\{ uv \mid (u,v) \in A(D) \}$.
An \emph{orientation} of a graph $G$
is a digraph having no directed $2$-cycles, no loops, and no multiple arcs
whose underlying graph  is $G$.
A \emph{tournament} is an orientation of a complete graph.
A \emph{$k$-partite tournament} is  an orientation of a complete $k$-partite graph for some positive integer $k \geq 2$.

The competition graphs of tournaments and those of bipartite tournaments
have been actively studied
%%여러 개를 나열할 때는 순서를 맞추어서 쓸 것
(see \cite{cho2002domination}, \cite{choi20171},  \cite{eoh2019niche},
  \cite{eoh2020m},
 \cite{factor2007domination},
  \cite{fisher2003domination},
  \cite{fisher1998domination}, and \cite{kim2016competition}
for papers related to this topic).

Recently, the authors of this paper began to study competition graphs of $k$-partite tournaments for $k \geq 2 $ and figured out
 the sizes of partite sets of multipartite tournaments whose competition graphs are complete~\cite{choi2022competitively}.
 
In this paper,
following up those results, we study triangle-free graphs which are competition graphs of multipartite tournaments.
we show that
 a connected triangle-free graph is the competition graph of a $k$-partite tournament if and only if $k \in \{3,4,5\}$, and list all the connected triangle-free graphs which are competition graphs of multipartite tournaments (Theorem~\ref{thm:complete-triangle-free-multipartite}).

We also show that
a disconnected triangle-free graph is the competition graph of a $k$-partite tournament if and only if $k \in \{2,3,4\}$, and list all the disconnected triangle-free graphs which are competition graphs of multipartite tournaments (Theorems~\ref{thm:charact-nonconnected-2-partite}, ~\ref{thm:charact-nonconnected-4-partite}, and~\ref{thm:charact-nonconnected-3-partite}).
%%k에 대한 조건들을 다시 언급하기에는 적절하지 않아서 multi-partite로..bipartite에서는 disconnected만 나오기에, 사실상 connected triangle-free는 해결한 것.
\section{The connected triangle-free competition graphs of multipartite tournaments}\label{property}
    \begin{Lem}\label{lem9} Let $D$ be an orientation of $K_{n_1,n_2,n_3}$ whose competition graph has no isolated vertex for some positive integers $n_1,n_2$, and $n_3$. Then at least two of $n_1,n_2$, and $n_3$ are greater than $1$. \end{Lem}
        \begin{proof} Suppose, to the contrary, that at most one of $n_1,n_2$, and $n_3$ is greater than $1$, that is, at least two of $n_1,n_2$, and $n_3$ equal $1$. Without loss of generality, we may assume that $n_1=n_2=1$.
        Let $\{u\}, \{v\}$, and $V$ be the partite sets of $D$ with $|V|=n_3$.
        Since $D$ is an orientation of $K_{n_1,n_2,n_3}$, either $(u,v) \in A(D)$ or $(v,u) \in A(D)$.
        By symmetry, we may assume that $(u,v) \in A(D)$.
        Since $C(D)$ has no isolated vertex, $v$ is adjacent to some vertex.
        Since $(u,v) \in A(D)$, $v$ is not adjacent to any vertex in $V$.
        Thus $u$ and $v$ are adjacent in $C(D)$ and so $u$ and $v$ have a common out-neighbor $w$ in $V$.
        Then neither $u$ nor $v$ is an out-neighbor of $w$ and so $N^+(w)=\emptyset$.
        Therefore $w$ is isolated in $C(D)$, which is a contradiction.
        \end{proof}
\begin{Lem} \label{lem:number-edges-triangle-free}
Let $D$ be a digraph with $n$ vertices for a positive integer $n$.
If the competition graph $C(D)$ of $D$ is  triangle-free, then
 $|E(C(D))| \leq |A(D)| / 2 \leq |V(D)|$.
\end{Lem}
\begin{proof}
Suppose that the competition graph $C(D)$ of $D$ is triangle-free.
Then $d^-(v) \leq 2$ for each vertex $v$ in $D$.
Therefore \[ |E(C(D))| \leq |\{v \in V(D) \mid d^-(v)=2  \}| \leq \frac{|A(D)|}{2}\leq |V(D)|. \]
\end{proof}
\begin{Lem} \label{lem:no-ori-1,2,4-triangle}
There is no orientation of $K_{4,2,1}$ whose competition graph is triangle-free.
\end{Lem}

\begin{proof}
Suppose, to the contrary, that
there exists an orientation $D$ of $K_{4,2,1}$ whose competition graph is triangle-free.
Then $|A(D)|=14$.
Since $C(D)$ is triangle-free,
each vertex has indegree at most $2$.
Then, since $|V(D)|=7$ and $|A(D)|=14$,
\begin{equation} \label{eq:lem:no-ori-1,2,4-triangle}
d^-(v)=2
\end{equation}
for each vertex $v$ in $D$.
Let $V_1=\{x_1,x_2,x_3,x_4\}$, $V_2=\{y_1,y_2\}$, and $V_3=\{z\}$ be the partite sets of $D$.
By~\eqref{eq:lem:no-ori-1,2,4-triangle},
each vertex in $V_1$ is a common out-neighbor of two vertices in $V_2 \cup V_3$ and so has outdegree $1$.
We note that if a vertex $a$ in $V_2 \cup V_3$ is an out-neighbor of a vertex $b$ in $V_1$, then $b$ is a common out-neighbor of the two vertices in $V_2 \cup V_3 \setminus \{a\}$ and so they are adjacent in $C(D)$.
Therefore there must be a vertex in $V_2 \cup V_3 $ which is not an out-neighbor of any vertex in $V_1$ to prevent from creating a triangle $y_1y_2z$ in $C(D)$.
By~\eqref{eq:lem:no-ori-1,2,4-triangle}, such a vertex in $V_2 \cup V_3 $ must be $z$ and \[N^-(z)=\{y_1,y_2\}.\]
Then $N^+(z)=V_1$.
By~\eqref{eq:lem:no-ori-1,2,4-triangle} again,
each of $y_1$ and $y_2$ is a common out-neighbor of two vertices in $V_1$.
Since each vertex in $V_1$ has outdegree $1$,
$N^-(y_1) \cap N^-(y_2) = \emptyset$.
Without loss of generality,
we may assume $N^-(y_1)=\{x_1,x_2\}$ and $N^-(y_2)=\{x_3,x_4\}$.
Then
$N^+(x_1)=N^+(x_2)=\{y_1\}$ and $N^+(x_3)=N^+(x_4)=\{y_2\}$, so \[N^-(x_1)=\{y_2,z\} \text{ and } N^-(x_3)=\{y_1,z\}.\]
Hence $\{y_1,y_2,z\}$ forms a triangle in $C(D)$, which is a contradiction.
\end{proof}
%%or에서는 respectively를 쓸 필요가 없다.

    \begin{Lem} \label{lem:sizes-of-3partite}
    Let $n_1$, $n_2$, and $n_3$ be positive integers such that $n_1 \geq n_2 \geq n_3$.
     Suppose that there exists an orientation $D$ of $K_{n_1,n_2,n_3}$ whose competition graph $C(D)$ is  triangle-free. Then
     one of the following holds:  (a) $n_1=n_2=n_3=2$;
        (b) $n_1 \leq 3$, $n_2=2$, and $n_3=1$;
        (c) $n_2=n_3=1$.
 In particular, if $C(D)$ is connected, then the case (c) does not occur.
 %repeat 부분이 있어서 안 되는 case를 서술하는 것으로 변경함.
    \end{Lem}

        \begin{proof} It is easy to check that
         $|A(D)|=n_1 n_2 + n_2 n_3 + n_3 n_1$.
     Then, by Lemma~\ref{lem:number-edges-triangle-free},
        \begin{eqnarray*}
          n_1 n_2 + n_2 n_3 + n_3 n_1 \leq 2(n_1 + n_2 + n_3).
        \end{eqnarray*}
        Thus
        \begin{eqnarray}\label{lem10:eq3}
          n_1(n_2-2)+n_2(n_3-2)+n_3(n_1-2) \leq 0
        \end{eqnarray}
        and so at least one of $n_1-2$, $n_2-2$, and $n_3-2$ is nonpositive.
        Since $n_1 \geq n_2 \geq n_3$, $n_3-2 \leq 0$.
        Suppose $n_3=2$.
        Then, by \eqref{lem10:eq3},
        $n_1n_2 \leq 4$ and so $(n_1,n_2,n_3)=(2,2,2)$.
        Now we suppose $n_3=1$. Then, by \eqref{lem10:eq3}, $ (n_1-1)(n_2-1)\leq 3$.
        Since $n_1 \geq n_2$,
        $n_2 \leq 2$.
        Suppose $n_2=2$. Then $n_1 \leq 4$.
        If $n_1 =4$, then $(n_1,n_2,n_3)=(4,2,1)$, which contradicts Lemma~\ref{lem:no-ori-1,2,4-triangle}.
        Therefore $n_1 \leq 3$ and so (b) holds. If $n_2=1$, then $n_3=1$ and so (c) holds.
If $C(D)$ is connected, then $C(D)$ has no isolated vertices and so none of $n_1$ and $n_2$ equals $1$ by Lemma~\ref{lem9} and so the ``in particular" part is true.
        \end{proof}
The following lemma is an immediate consequence of Lemma~\ref{lem:sizes-of-3partite}.
         \begin{Lem}\label{cor13-2} For a connected triangle-free graph $G$
         of order $n$,
          if $G$ is the competition graph of a tripartite tournament, then
          $n \in \{5,6\}$. \end{Lem}

    \begin{Lem}\label{thm12} For a positive integer $n \geq3$, a cycle $C_n$ of length $n$ is the competition graph of a tripartite tournament if and only if $n=6$. \end{Lem}
        \begin{proof} Let $D$ be the digraph in Figure~\ref{fig2} which is an orientation of $K_{2,2,2}$.
\begin{figure}
\begin{center}
\begin{tikzpicture}[auto,thick, scale=1.2]
    \tikzstyle{player}=[minimum size=5pt,inner sep=0pt,outer sep=0pt,draw,circle]
    \tikzstyle{source}=[minimum size=5pt,inner sep=0pt,outer sep=0pt,ball color=black, circle]
    \tikzstyle{arc}=[minimum size=5pt,inner sep=1pt,outer sep=1pt, font=\footnotesize]
    \draw (270:1.4cm) node (name) {$D$};
    \path (0:1cm)    node [player]  (x1) {};
    \path (60:1cm)   node [player]  (x2) {};
    \path (120:1cm)   node [player]  (y1) {};
    \path (180:1cm)   node [player]  (y2) {};
    \path (240:1cm)   node [player]  (z1) {};
    \path (300:1cm)   node [player]  (z2) {};

    \foreach \tail/\head in {z/x,x/y,y/z}{
        \foreach \i in {1,2}{
            \draw[black,thick,-stealth] (\tail\i) - +(\head1);%
        }
    }

    \foreach \n/\a/\m in {z/x/y,x/y/z,y/z/x}{
        \draw[black,thick,-stealth] (\n2) - +(\m2);%
        \draw[black,thick,-stealth] (\a1) - +(\m2);%
    }
\end{tikzpicture}
\hspace{5em}
\begin{tikzpicture}[auto,thick, scale=1.2]
    \tikzstyle{player}=[minimum size=5pt,inner sep=0pt,outer sep=0pt,draw,circle]
    \tikzstyle{source}=[minimum size=5pt,inner sep=0pt,outer sep=0pt,ball color=black, circle]
    \tikzstyle{arc}=[minimum size=5pt,inner sep=1pt,outer sep=1pt, font=\footnotesize]

    \draw (270:1.4cm) node (name) {$C(D)$};
    \path (0:1cm)    node [player]  (1) {};
    \path (60:1cm)   node [player]  (2) {};
    \path (120:1cm)   node [player]  (3) {};
    \path (180:1cm)   node [player]  (4) {};
    \path (240:1cm)   node [player]  (5) {};
    \path (300:1cm)   node [player]  (6) {};

    \draw (1)--(2)--(3)--(4)--(5)--(6)--(1);
\end{tikzpicture}
\caption{A digraph $D$ which is an orientation of $K_{2,2,2}$ and whose competition graph is isomorphic to $C_6$}\label{fig2}
\end{center}
\end{figure}
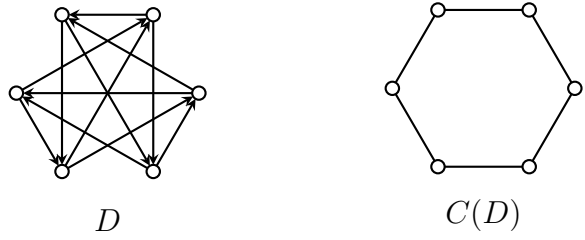
        It is easy to check that $C(D) \cong C_6$. Therefore the ``if'' part is true.

        Now suppose that a cycle $C_n$ is the competition graph of a tripartite tournament $T$ for a positive integer $n \geq 3$..
        If $n=3$, then the only possible size of partite sets is $(1,1,1)$ which is impossible by Lemma~\ref{lem9}. Therefore $n \geq 4$.
        Thus
        $n=5$ or $n=6$ by
        Lemma~\ref{cor13-2}.
        Suppose, to the contrary, that $n=5$.
        Then $T$ is an orientation of $K_{2,2,1}$ by Lemma~\ref{lem:sizes-of-3partite}.
        Since $C(T)$ is triangle-free, $d^-(v) \leq 2$ for each $v \in V(T)$.
        Moreover, since each edge is a maximal clique and there are five edges in $C(T)$, each vertex has indegree $2$ in $T$.
        Therefore
        \begin{eqnarray*}
           8 = |A(T)| = \sum_{v \in V(T)} d^-(v) = 10
        \end{eqnarray*}
        and we reach a contradiction. Thus $n=6$.
        \end{proof}
    \begin{Lem}\label{thm13} For a positive integer $n \geq3$, a path $P_n$ of length $n-1$ is the competition graph of a tripartite tournament if and only if $n=6$. \end{Lem}
        \begin{proof} Let $D$ be the digraph in Figure~\ref{fig3} which is an orientation of $K_{3,2,1}$.
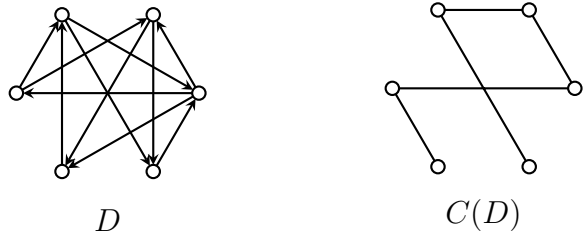
\begin{figure}
\begin{center}
\begin{tikzpicture}[auto,thick, scale=1.2]
    \tikzstyle{player}=[minimum size=5pt,inner sep=0pt,outer sep=0pt,draw,circle]
    \tikzstyle{source}=[minimum size=5pt,inner sep=0pt,outer sep=0pt,ball color=black, circle]
    \tikzstyle{arc}=[minimum size=5pt,inner sep=1pt,outer sep=1pt, font=\footnotesize]

    \draw (270:1.4cm) node (name) {$D$};
    \path (0:1cm)    node [player]  (x1) {};
    \path (60:1cm)   node [player]  (y1) {};
    \path (120:1cm)   node [player]  (y2) {};
    \path (180:1cm)   node [player]  (z1) {};
    \path (240:1cm)   node [player]  (z2) {};
    \path (300:1cm)   node [player]  (z3) {};

    \draw[black,thick,-stealth] (x1) - + (y1);
    \draw[black,thick,-stealth] (x1) - + (z1);
    \draw[black,thick,-stealth] (x1) - + (z2);
    \draw[black,thick,-stealth] (y2) - + (x1);
    \draw[black,thick,-stealth] (z3) - + (x1);
    \draw[black,thick,-stealth] (y1) - + (z2);
    \draw[black,thick,-stealth] (y1) - + (z3);
    \draw[black,thick,-stealth] (z1) - + (y1);
    \draw[black,thick,-stealth] (z1) - + (y2);
    \draw[black,thick,-stealth] (z2) - + (y2);
    \draw[black,thick,-stealth] (y2) - + (z3);
\end{tikzpicture}
\hspace{5em}
\begin{tikzpicture}[auto,thick, scale=1.2]
    \tikzstyle{player}=[minimum size=5pt,inner sep=0pt,outer sep=0pt,draw,circle]
    \tikzstyle{source}=[minimum size=5pt,inner sep=0pt,outer sep=0pt,ball color=black, circle]
    \tikzstyle{arc}=[minimum size=5pt,inner sep=1pt,outer sep=1pt, font=\footnotesize]

    \draw (270:1.4cm) node (name) {$C(D)$};

    \path (0:1cm)    node [player]  (1) {};
    \path (60:1cm)   node [player]  (2) {};
    \path (120:1cm)   node [player]  (3) {};
    \path (180:1cm)   node [player]  (4) {};
    \path (240:1cm)   node [player]  (5) {};
    \path (300:1cm)   node [player]  (6) {};
    
     \draw (6)--(3)--(2)--(1)--(4)--(5);
\end{tikzpicture}
\caption{A digraph $D$ which is an orientation of $K_{3,2,1}$ and whose competition graph is isomorphic to $P_6$}\label{fig3}
\end{center}
\end{figure}
        It is easy to check that $C(D) \cong P_6$. Therefore the ``if'' part is true.

        Now suppose that a path $P_n$ is the competition graph of a tripartite tournament $T$ for a positive integer $n \geq 3$.
        Since $P_n$ is connected and triangle-free, by Lemma~\ref{cor13-2}, $n \in \{5,6\}$.
        Suppose, to the contrary, that $n=5$.
        Then $T$ is an orientation of $K_{2,2,1}$ by Lemma~\ref{lem:sizes-of-3partite}.
        Let $V_1, V_2$, and $V_3$ be the partite sets with $|V_1|=|V_2|=2$ and $|V_3|=1$.
        Since $C(T)$ is triangle-free, $d^-(v) \leq 2$ for each $v \in V(T)$.
        Since there are four edges in $C(T)$, there are four vertices of indegree $2$ in $T$.
        Since $|A(T)|=8$ and $n=5$, there exists exactly one vertex of indegree $0$ in $T$.
        Let $u$ be the vertex of indegree $0$ in $T$.
        If $V_3=\{u\}$, then $N^+(u)=V_1 \cup V_2$.
        Otherwise, either $N^+(u)=V_2\cup V_3$ or $N^+(u)=V_1 \cup V_3$.
        Therefore $d^+(u)=3$ or $4$.
        Thus $u$ is incident to at least three edges in $C(T)$, which is impossible on a path.
        Hence $n=6$.
        \end{proof}
%% 4/13 edges, vertices 앞에 수 (10이하)는 영어로 바꿔놓기
    \begin{Lem}\label{lem:K_{1,2,3}} Let $D$ be an orientation of $K_{3,2,1}$ whose competition graph is connected and triangle-free.
    Then the following are true:
    \begin{itemize}
        \item[(1)] There is no vertex of indegree $0$ in $D$;
        \item[(2)] There are exactly five vertices with indegree $2$ in $D$ no two of which have the same in-neighborhood.
    \end{itemize}
    \end{Lem}
        \begin{proof} Since $C(D)$ is triangle-free, $d^-(v) \leq 2$ for each $v \in V(D)$.
        If there is a vertex of indegree $0$ in $D$, then $11=\sum_{v \in V(D)} d^-(v) \leq 10$ and we reach a contradiction.
        Therefore the statement (1) is true and so the indegree sequence of $D$ is $(2,2,2,2,2,1)$.
        Meanwhile, since $C(D)$ is connected, the number of edges of $C(D)$ is at least $5$.
        Thus there are exactly five vertices with indegree $2$ in $D$ no two of which have the same in-neighborhood.
        \end{proof}

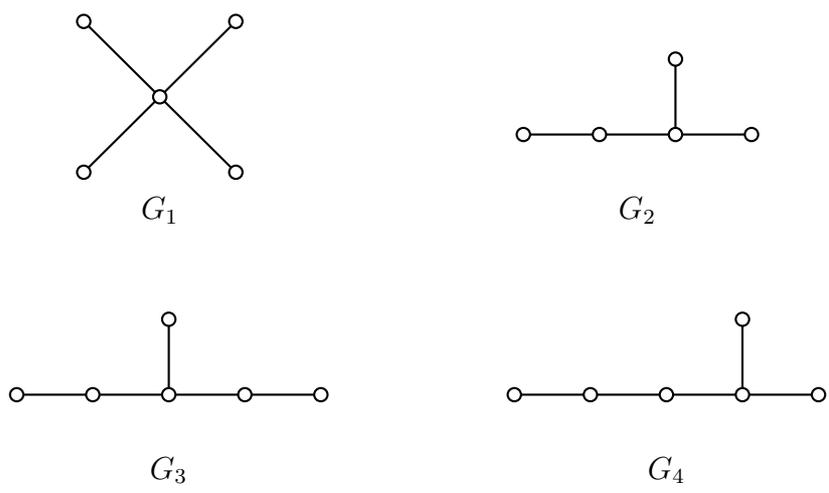
\begin{figure}
\begin{center}
\begin{tikzpicture}[auto,thick, scale=1.0]
    \tikzstyle{player}=[minimum size=5pt,inner sep=0pt,outer sep=0pt,draw,circle]
    \tikzstyle{source}=[minimum size=5pt,inner sep=0pt,outer sep=0pt,ball color=black, circle]
    \tikzstyle{arc}=[minimum size=5pt,inner sep=1pt,outer sep=1pt, font=\footnotesize]

    \draw (270:1.5cm) node (name) {$G_1$};

    \path (0,0)    node [player]  (1) {};
    \path (-1,1)   node [player]  (2) {};
    \path (-1,-1)   node [player]  (3) {};
    \path (1,1)   node [player]  (4) {};
    \path (1,-1)   node [player]  (5) {};

    \draw (1)--(2);
    \draw (1)--(3);
    \draw (1)--(4);
    \draw (1)--(5);
\end{tikzpicture}
\hspace{8em}
\begin{tikzpicture}[auto,thick, scale=1.0]
    \tikzstyle{player}=[minimum size=5pt,inner sep=0pt,outer sep=0pt,draw,circle]
    \tikzstyle{source}=[minimum size=5pt,inner sep=0pt,outer sep=0pt,ball color=black, circle]
    \tikzstyle{arc}=[minimum size=5pt,inner sep=1pt,outer sep=1pt, font=\footnotesize]

    \draw (270:1cm) node (name) {$G_2$};

    \path (0.5,0)    node [player]  (1) {};
    \path (-0.5,0)   node [player]  (2) {};
    \path (-1.5,0)   node [player]  (3) {};
    \path (1.5,0)   node [player]  (4) {};
    \path (0.5,1)   node [player]  (5) {};

    \draw (1)--(2)--(3);
    \draw (1)--(4);
    \draw (1)--(5);

\end{tikzpicture}

\vskip1cm
\begin{tikzpicture}[auto,thick, scale=1.0]
    \tikzstyle{player}=[minimum size=5pt,inner sep=0pt,outer sep=0pt,draw,circle]
    \tikzstyle{source}=[minimum size=5pt,inner sep=0pt,outer sep=0pt,ball color=black, circle]
    \tikzstyle{arc}=[minimum size=5pt,inner sep=1pt,outer sep=1pt, font=\footnotesize]

    \draw (270:1cm) node (name) {$G_3$};

    \path (0,0)    node [player]  (1) {};
    \path (-1,0)   node [player]  (2) {};
    \path (-2,0)   node [player]  (3) {};
    \path (1,0)   node [player]  (4) {};
    \path (2,0)   node [player]  (5) {};
    \path (0,1)   node [player]  (6) {};

    \draw (1)--(2)--(3);
    \draw (1)--(4)--(5);
    \draw (1)--(6);

\end{tikzpicture}
\hspace{5em}
\begin{tikzpicture}[auto,thick, scale=1.0]
    \tikzstyle{player}=[minimum size=5pt,inner sep=0pt,outer sep=0pt,draw,circle]
    \tikzstyle{source}=[minimum size=5pt,inner sep=0pt,outer sep=0pt,ball color=black, circle]
    \tikzstyle{arc}=[minimum size=5pt,inner sep=1pt,outer sep=1pt, font=\footnotesize]

    \draw (270:1cm) node (name) {$G_4$};

    \path (0,0)    node [player]  (1) {};
    \path (-1,0)   node [player]  (2) {};
    \path (-2,0)   node [player]  (3) {};
    \path (1,0)   node [player]  (4) {};
    \path (2,0)   node [player]  (5) {};
    \path (1,1)   node [player]  (6) {};

    \draw (1)--(2)--(3);
    \draw (1)--(4)--(5);
    \draw (4)--(6);

\end{tikzpicture}
\caption{Connected triangle-free graphs mentioned in Lemma~\ref{thm:tri-fre}}\label{fig:tri-fre}
\end{center}
\end{figure}
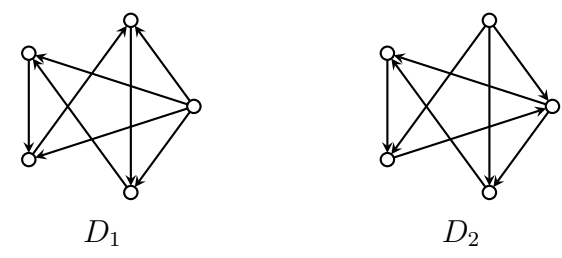
\begin{figure}
\begin{center}
\begin{tikzpicture}[auto,thick, scale=1.2]
    \tikzstyle{player}=[minimum size=5pt,inner sep=0pt,outer sep=0pt,draw,circle]
    \tikzstyle{source}=[minimum size=5pt,inner sep=0pt,outer sep=0pt,ball color=black, circle]
    \tikzstyle{arc}=[minimum size=5pt,inner sep=1pt,outer sep=1pt, font=\footnotesize]

    \draw (270:1.4cm) node (name) {$D_1$};

   \path (0:1cm)    node [player]  (x1) {};
    \path (72:1cm)   node [player]  (y1) {};
    \path (144:1cm)   node [player]  (y2) {};
    \path (216:1cm)   node [player]  (z1) {};
    \path (288:1cm)   node [player]  (z2) {};

    \draw[black,thick,-stealth] (x1) - + (y1);
    \draw[black,thick,-stealth] (x1) - + (y2);
    \draw[black,thick,-stealth] (x1) - + (z1);
    \draw[black,thick,-stealth] (x1) - + (z2);
    \draw[black,thick,-stealth] (y1) - + (z2);
    \draw[black,thick,-stealth] (y2) - + (z1);
    \draw[black,thick,-stealth] (z1) - + (y1);
    \draw[black,thick,-stealth] (z2) - + (y2);

\end{tikzpicture}
\hspace{5em}
\begin{tikzpicture}[auto,thick, scale=1.2]
    \tikzstyle{player}=[minimum size=5pt,inner sep=0pt,outer sep=0pt,draw,circle]
    \tikzstyle{source}=[minimum size=5pt,inner sep=0pt,outer sep=0pt,ball color=black, circle]
    \tikzstyle{arc}=[minimum size=5pt,inner sep=1pt,outer sep=1pt, font=\footnotesize]
     \draw (270:1.4cm) node (name) {$D_2$};
  \path (0:1cm)    node [player]  (x1) {};
    \path (72:1cm)   node [player]  (y1) {};
    \path (144:1cm)   node [player]  (y2) {};
    \path (216:1cm)   node [player]  (z1) {};
    \path (288:1cm)   node [player]  (z2) {};

    \draw[black,thick,-stealth] (x1) - + (y2);
    \draw[black,thick,-stealth] (x1) - + (z2);
    \draw[black,thick,-stealth] (y1) - + (x1);
    \draw[black,thick,-stealth] (y1) - + (z1);
    \draw[black,thick,-stealth] (y1) - + (z2);
    \draw[black,thick,-stealth] (y2) - + (z1);
    \draw[black,thick,-stealth] (z1) - + (x1);
    \draw[black,thick,-stealth] (z2) - + (y2);

\end{tikzpicture}
\caption{Two digraphs $D_1$ and $D_2$ which are orientations of $K_{2,2,1}$ and whose competition graphs are isomorphic to $G_1$ and $G_2$, respectively}\label{fig:n=5:tri-fre}
\end{center}
\end{figure}

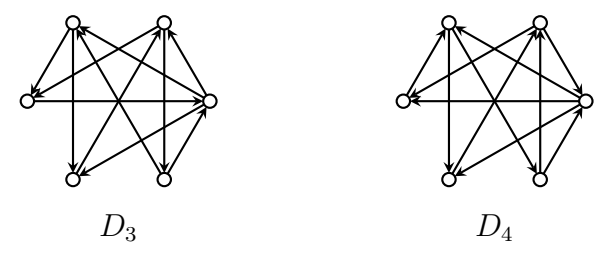
\begin{figure}
\begin{center}
\begin{tikzpicture}[auto,thick, scale=1.2]
    \tikzstyle{player}=[minimum size=5pt,inner sep=0pt,outer sep=0pt,draw,circle]
    \tikzstyle{source}=[minimum size=5pt,inner sep=0pt,outer sep=0pt,ball color=black, circle]
    \tikzstyle{arc}=[minimum size=5pt,inner sep=1pt,outer sep=1pt, font=\footnotesize]

     \draw (270:1.4cm) node (name) {$D_3$};
       \path (0:1cm)   node [player]  (x1) {};
    \path (60:1cm)   node [player]  (y1) {};
    \path (120:1cm)   node [player]  (y2) {};
    \path (180:1cm)    node [player]  (z1) {};
    \path (240:1cm)   node [player]  (z2) {};
      \path (300:1cm)   node [player]  (z3) {};

    \draw[black,thick,-stealth] (x1) - + (y1);
    \draw[black,thick,-stealth] (x1) - + (y2);
    \draw[black,thick,-stealth] (x1) - + (z2);
    \draw[black,thick,-stealth] (y1) - + (z1);
    \draw[black,thick,-stealth] (y1) - + (z3);
    \draw[black,thick,-stealth] (y2) - + (z1);
    \draw[black,thick,-stealth] (y2) - + (z2);
    \draw[black,thick,-stealth] (z1) - + (x1);
    \draw[black,thick,-stealth] (z2) - + (y1);
    \draw[black,thick,-stealth] (z3) - + (x1);
    \draw[black,thick,-stealth] (z3) - + (y2);
\end{tikzpicture}
\hspace{5em}
\begin{tikzpicture}[auto,thick, scale=1.2]
    \tikzstyle{player}=[minimum size=5pt,inner sep=0pt,outer sep=0pt,draw,circle]
    \tikzstyle{source}=[minimum size=5pt,inner sep=0pt,outer sep=0pt,ball color=black, circle]
    \tikzstyle{arc}=[minimum size=5pt,inner sep=1pt,outer sep=1pt, font=\footnotesize]

    \draw (270:1.4cm) node (name) {$D_4$};
    \path (0:1cm)   node [player]  (x1) {};
    \path (60:1cm)   node [player]  (y1) {};
    \path (120:1cm)   node [player]  (y2) {};
    \path (180:1cm)    node [player]  (z1) {};
    \path (240:1cm)   node [player]  (z2) {};
      \path (300:1cm)   node [player]  (z3) {};

    \draw[black,thick,-stealth] (x1) - + (y2);
    \draw[black,thick,-stealth] (x1) - + (z1);
    \draw[black,thick,-stealth] (x1) - + (z2);
    \draw[black,thick,-stealth] (y1) - + (x1);
    \draw[black,thick,-stealth] (y1) - + (z1);
    \draw[black,thick,-stealth] (y2) - + (z2);
    \draw[black,thick,-stealth] (y2) - + (z3);
    \draw[black,thick,-stealth] (z1) - + (y2);
    \draw[black,thick,-stealth] (z2) - + (y1);
    \draw[black,thick,-stealth] (z3) - + (y1);
    \draw[black,thick,-stealth] (z3) - + (x1);

\end{tikzpicture}
\caption{Two digraphs $D_3$ and $D_4$ which are orientations of $K_{3,2,1}$ and whose competition graphs are isomorphic to $G_3$ and $G_4$, respectively}\label{fig:n=6:tri-fre}
\end{center}
\end{figure}

    \begin{Lem}\label{thm:tri-fre} Let $G$ be a connected and triangle-free graph with $n$ vertices.
    Then $G$ is the competition graph of a tripartite tournament if and only if $G$ is isomorphic to a graph belonging to the following set:
    \[
     \begin{cases}\{ G_1, G_2 \} & \mbox{if $n=5$;} \\
     \{ G_3, G_4, P_6, C_6  \}  & \mbox{if $n=6$} \end{cases}
     \]
     where $G_i$ is the graph given in Figure~\ref{fig:tri-fre} for each $1\leq i \leq 4$.
    \end{Lem}
        \begin{proof} Let $D$ be a tripartite tournament whose competition graph is $G$.
        %%Competition graph $G$ 에서 is G로 고침 G가 이미 주어져있기 때문에.
        Since $G$ is triangle-free,
        \begin{eqnarray}\label{equation:indegree2}
          d^-(v) \leq 2
        \end{eqnarray}
        for each $v \in V(D)$ and $n \in \{5,6\}$ by Lemma~\ref{cor13-2}.
        If $G$ is a path or a cycle, then, by Lemmas~\ref{thm12} and~\ref{thm13}, $G$ is isomorphic to $P_6$ or $C_6$.
        Now we suppose that $G$ is neither a path nor a cycle.
        Then, there exists a vertex of degree at least three in $C(D)$.

        {\it Case 1.} $n=5$. Then, by Lemma~\ref{lem:sizes-of-3partite}, $D$ is an orientation of $K_{2,2,1}$. Since $|A(D)|=8$ and $C(D)$ is connected, by \eqref{equation:indegree2}, there are exactly four edges in $C(D)$. Therefore $C(D)$ is isomorphic to $G_1$ or $G_2$ in Figure~\ref{fig:tri-fre}. Thus the ``only if" part is true in this case. To show the ``if" part, let $D_1$ and $D_2$ be the digraphs in Figure~\ref{fig:n=5:tri-fre} which are some orientations of $K_{2,2,1}$. It is easy to check that $C(D_1) \cong G_1$ and $C(D_2) \cong G_2$. Hence the ``if" part is true.

        {\it Case 2.} $n = 6$.
        Then, by Lemma~\ref{lem:sizes-of-3partite}, $D$ is an orientation of $K_{3,2,1}$ or $K_{2,2,2}$.
        Suppose that $D$ is an orientation of $K_{2,2,2}$.
        Since $\sum_{v \in V(D)} d^-(v)=12$, by \eqref{equation:indegree2}, $d^-(v)=2$ for each $v \in V(D)$ and so $d^+(v)=2$ for each $v \in V(D)$.
        Therefore every vertex has degree at most $2$ in $C(D)$, which is a contradiction to the assumption that $G$ is neither a path nor cycle.
        Thus $D$ is an orientation of $K_{3,2,1}$.
        By Lemma~\ref{lem:K_{1,2,3}}, there are exactly five edges in $C(D)$.
        Let $V_1$, $V_2$, and $V_3$ be the partite sets of $D$ with $|V_i|=i$ for each $i=1,2,$ and $3$.

        Suppose that there is a vertex $w$ of degree at least $4$ in $C(D)$.
        Then, by \eqref{equation:indegree2}, $w$ has outdegree at least $4$ in $D$.
        Thus $w$ belongs to $V_1$ or $V_2$.
        If $w$ belongs to $V_2$, then the indegree of $w$ is $0$, which contradicts Lemma~\ref{lem:K_{1,2,3}}(1).
        Therefore $w \in V_1$ and so
        $V_1=\{w\}$.
        Moreover, the outdegree of $w$ in $D$ is $4$ by Lemma~\ref{lem:K_{1,2,3}}(1).
        Then the indegree of each vertex in $D$ except $w$ is exactly $2$ by Lemma~\ref{lem:K_{1,2,3}}(2).
        If three out-neighbors of $w$ belong to the same partite set, then two of them share the same in-neighborhood, which contradicts Lemma~\ref{lem:K_{1,2,3}}(2).
        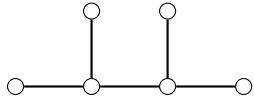
\begin{figure}
\begin{center}
\begin{tikzpicture}[x=1.0cm, y=1.0cm]
  \vertex (x1) at (0,0) [label=above:$$]{};
  \vertex (x2) at (1,0) [label=above:$$]{};
  \vertex (x3) at (-1,0) [label=above:$$]{};
  \vertex (x4) at (0,1) [label=above:$$]{};
    \vertex (x5) at (-1,1) [label=above:$$]{};
    \vertex (x6) at (-2,0) [label=above:$$]{};
   \path
   (x5) edge [-,thick] (x3)
   (x1) edge [-,thick] (x2)
   (x1) edge [-,thick] (x3)
   (x1) edge [-,thick] (x4)
(x6) edge [-,thick] (x3)
;
%\draw (-0.5, -0.8) node{$G^*$};
\end{tikzpicture}
 \end{center}
 \caption{A graph considered in the proof of Lemma~\ref{thm:tri-fre}}
\label{fig:example-k=3,2,1-competition}
\end{figure}
        Therefore two of the out-neighbors of $w$ belong to $V_2$ and the remaining out-neighbors belong to $V_3$.
        Since the indegree of each vertex in $D$ except $w$ is exactly $2$ by Lemma~\ref{lem:K_{1,2,3}}(2),
         each vertex in $V_2$ has exactly one in-neighbor in $V_3$.
        Thus there is one vertex in $V_3$ which is not an in-neighbor of any vertex in $V_2$.
        Then $w$ is the only its out-neighbor and so it is isolated in $C(D)$.
        Hence we have reached a contradiction and so the degree of each vertex of $C(D)$ is at most $3$.

        Now suppose that there are at least two vertices $x$ and $y$ of degree $3$ in $C(D)$.
        Then, since the number of edges in $C(D)$ is exactly $5$, $C(D)$ is isomorphic to the tree given in Figure~\ref{fig:example-k=3,2,1-competition}. By \eqref{equation:indegree2}, $d^+(x) \geq 3$ and $d^+(y) \geq 3$.
        If $x$ or $y$ belongs to $V_3$, then $d^-(x)=0$ or $d^-(y)=0$, which contradicts Lemma~\ref{lem:K_{1,2,3}}(1).
        Therefore $x$ and $y$ belong to $V_1$ or $V_2$.
        Suppose that $V_2=\{x,y\}$.
        Then, since $|V_1 \cup V_3|=4$, there are at least two vertices of indegree $2$ in $V_1 \cup V_3$ which have the same in-neighborhood $\{x,y\}$, which contradicts Lemma~\ref{lem:K_{1,2,3}}(2).
        Thus one of $x$ and $y$ belongs to $V_1$ and the other belongs to $V_2$.
        Without loss of generality, we may assume that $x\in V_1$ and $y \in V_2$.
        Then $V_1=\{x\}$ and, by Lemma~\ref{lem:K_{1,2,3}}(1), $d^-(y) \neq 0$, so $d^+(y)=3$.
        If $N^+(y)=V_3$, then $y$ is adjacent to at most two vertices in $C(D)$, which is a contradiction.
        Therefore $N^+(y) \cap V_3=\{z_1,z_2\}$ for some vertices $z_1$ and $z_2$ in $V_3$ and $(y,x) \in A(D)$.
        Since $C(D)$ is isomorphic to the tree given in Figure~\ref{fig:example-k=3,2,1-competition},
        $x$ and $y$ have a common out-neighbor in $V_3$.
        By Lemma~\ref{lem:K_{1,2,3}}(2), exactly one of $z_1$ and $z_2$ can be a common out-neighbor of $x$ and $y$ in $D$.
        By symmetry, we may assume that $z_1$ is a common out-neighbor of $x$ and $y$ and $(z_2, x) \in A(D)$.
        Then $N^+(x)=\{y',z_1,z_3\}$ for the vertices $y'$ other than $y$ in $V_2$ and $z_3$ other than $z_1$ and $z_2$ in $V_3$.
        Therefore, by \eqref{equation:indegree2}, $(z_1,y') \in A(D)$ and so, by \eqref{equation:indegree2} again, $N^+(y')=\{z_2,z_3\}$.
        Then $y'$ is adjacent to $y$ and $x$ in $C(D)$ and so $\{x,y,y'\}$ forms a triangle in $C(D)$, which is a contradiction.
        Thus we have shown that there is the only one vertex of degree $3$ in $C(D)$ and so $C(D)$ is isomorphic to $G_3$ or $G_4$ in Figure~\ref{fig:tri-fre}. Hence the ``only if" part is true. To show the ``if" part is true, let $D_3$ and $D_4$ be two digraphs given in Figure~\ref{fig:n=6:tri-fre} which are isomorphic to some
        orientations of $K_{3,2,1}$. It is easy to check that $C(D_3) \cong G_3$ and $C(D_4) \cong G_4$.
        Hence the ``if" part is true. \end{proof}
The following lemma is immediately true by the definition of the competition graph.
\begin{Lem} \label{lem:subdigraph}
Let $D$ be a digraph and $D'$ be a subdigraph of $D$.
Then the competition graph of $D'$ is a subgraph of the competition graph of $D$.
\end{Lem}

\begin{Lem} \label{lem:connected-k-condition}
For a positive integer $k \geq 6$,
each competition graph of a $k$-partite tournament  contains a triangle.
\end{Lem}

\begin{proof}
Suppose that $D$ is a $k$-partite tournament for a positive integer $k \geq 6$.
Let $V_1, V_2, \ldots, V_k$ be the partite sets of $D$.
Then we take a vertex $v_i$ in $V_i$ for each $1\leq i \leq 6$.
Then $\{v_1,\ldots,v_6\}$ forms a $6$-tournament $T$.
Since $T$ has $15$ arcs,
there exists a vertex in $T$ whose indegree is at least $3$.
Therefore $C(T)$ has a triangle and so, by Lemma~\ref{lem:subdigraph},
$C(D)$ contains a triangle.
\end{proof}

\begin{Prop} \label{prop:minimum-number-of-edges}
(Fisher~\cite{fisher1998domination}) For $n  \geq 2$, the minimum possible number of edges in the competition graph of an $n$-tournament is $ \binom{n}{2} - n$.
\end{Prop}
An $n$-tournament is {\it regular} if $n$ is odd and every vertex has outdegree $(n -1)/2$.
Fisher~\cite{fisher1995domination} and Cho~\cite{cho2002domination} showed that a path on four or more vertices is not
the domination graph of a tournament and that the domination graph of a regular $n$-tournament $(n \geq 3)$ is either an odd cycle or a forest of two or more paths, respectively.
Here, the {\it domination graph } of a tournament $T$ is the complement of the competition graph of the tournament formed by reversing the arcs of $T$.
%Here 를 쓴 것 <- domination이 앞에서는 정의 없었고, 여기서 나왔기 때문에// 부드럽게 연결
Accordingly, their results can be restated as follows.

\begin{Prop}
\label{prop:no-path-on=four-tournament}
(Fisher~\cite{fisher1995domination})
A path on four or more vertices is not
the complement of the competition graph of a tournament.
\end{Prop}

\begin{Prop} \label{prop:domination-regular-tournament}
(Cho~\cite{cho2002domination}) If $T$ is a regular $n$-tournament $(n \geq 3)$, then the complement of the competition graph of $T$ is either an odd cycle or a forest of two or more paths.
\end{Prop}

\begin{Lem} \label{lem:triangle-free-5-partite-tournament}
If the competition graph $C(D)$ of a $5$-partite tournament $D$ is triangle-free, then $D$ is a regualr $5$-tournament and $C(D)$ is isomorphic to a cycle of length $5$.
\end{Lem}

\begin{proof}
Suppose that $D$ is a $5$-partite tournament whose competition graph is triangle-free.
Let $V_1,\ldots,V_5$ be the partite sets of $D$.
To show $|V(D)|=5$ by contrary, suppose
$|V(D)|\geq 6$.
Then there exists a partite set whose size is at least $2$.
  Without loss of generality, we may assume $|V_1| \geq 2$. We take $v_i$ in $V_i$ for each $1\leq i \leq 5$.
 Then we may take a vertex $v_{1}'$ distinct from $v_1$ in $V_1$ so that the subdigraph $T$ induced by $\{v_1,v'_1,v_2,\ldots,v_5\}$ is a $5$-partite tournament.
  Since $T$ has $14$ arcs and $|V(T)|=6$,
  there exists a vertex of indegree at least $3$ in $T$, which is a contradiction.
  Thus $V_i=\{v_i\}$ for each $1\leq i \leq 5$.
  Then $D$ is a tournament.

  Since $|V(D)|=5$,
   $|E(C(D))| \geq 5$ by Proposition~\ref{prop:minimum-number-of-edges} and so, by Lemma~\ref{lem:number-edges-triangle-free}, we have $|E(C(D))|=5$.
   Then, since $|V(D)|=5$,
   each vertex has indegree exactly $2$
   and so each vertex has outdegree $2$.
   Thus $D$ is a regular $5$-tournament.
   Since it is easy to check that a regular $5$-tournament is unique up to isomorphism as shown in Figure~\ref{fig:n=5:tri-fre-1},
   $C(D)$ is isomorphic to a cycle of length $5$.
    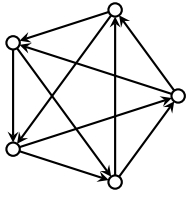
\begin{figure}
\begin{center}
\begin{tikzpicture}[auto,thick, scale=1.2]
    \tikzstyle{player}=[minimum size=5pt,inner sep=0pt,outer sep=0pt,draw,circle]
    \tikzstyle{source}=[minimum size=5pt,inner sep=0pt,outer sep=0pt,ball color=black, circle]
    \tikzstyle{arc}=[minimum size=5pt,inner sep=1pt,outer sep=1pt, font=\footnotesize]
%    \draw (270:2cm) node (name) {$D_6$};
   \path (0:1cm)   node [player]  (v1) {};
    \path (72:1cm)   node [player]  (v2) {};
    \path (144:1cm)   node [player]  (v3) {};
    \path (216:1cm)    node [player]  (v4) {};
    \path (288:1cm)   node [player]  (v5) {};

    \draw[black,thick,-stealth] (v1) - + (v2);
    \draw[black,thick,-stealth] (v2) - + (v3);
    \draw[black,thick,-stealth] (v3) - + (v4);
    \draw[black,thick,-stealth] (v4) - + (v5);
    \draw[black,thick,-stealth] (v5) - + (v1);
    \draw[black,thick,-stealth] (v1) - + (v3);
    \draw[black,thick,-stealth] (v2) - + (v4);
    \draw[black,thick,-stealth] (v3) - + (v5);
    \draw[black,thick,-stealth] (v4) - + (v1);
    \draw[black,thick,-stealth] (v5) - + (v2);
\end{tikzpicture}
\caption{A regular $5$-tournament}\label{fig:n=5:tri-fre-1}
\end{center}
\end{figure}
\end{proof}

\begin{Lem} \label{lem:same-out-in-neighbor-same-partite}
Let $D$ be a multipartite tournament whose competition graph is triangle-free.
If two vertices $u$ and $v$ with outdegree at least one have the same out-neighborhood or in-neighborhood,
 then $u$ and $v$ belong to the same partite set of $D$ and form a component in $C(D)$.
\end{Lem}
%%기존의 문장은 어색해서 다시 씀. 기호 없애고, resp으로 변경.

\begin{proof}
Suppose, to the contrary, that there are two vertices $u$ and $v$ with outdegree at least one such that $N^+(u)=N^+(v)$ or $N^-(u)=N^-(v)$ but $u$ and $v$ belong to the distinct partite sets.
Then $(u,v) \in A(D)$ or $(v,u) \in A(D)$.
Without loss of generality,
we may assume
$(u,v) \in A(D)$.
Then $u \in N^-(v)$ but $u \notin N^-(u)$.
Therefore $N^-(u) \neq N^-(v)$ and $N^+(u) \neq N^+(v)$, which is a contradiction.
Thus $u$ and $v$ belong to the same partite set.
Then, since $D$ is a multipartite tournament, $N^+(u)=N^+(v)$ if and only if $N^-(u)=N^-(v)$.
Therefore $N^+(u)=N^+(v)\neq \emptyset$ by the hypothesis.
Since $C(D)$ is triangle-free, $u$ and $v$ are the only in-neighbors of each vertex in $N^+(u)$ and so they form a component in $C(D)$.
\end{proof}

\begin{Lem}  \label{lem:condition-of-sizes-4-partite sets}
Let $n_1,n_2,n_3,n_4$ be positive integers such that $n_1 \geq \cdots \geq n_4$.
If $D$ is an orientation of $K_{n_1,n_2,n_3,n_4}$ whose competition graph is triangle-free, then $n_1 \leq 2$ and
$n_2=n_3=n_4=1$.
\end{Lem}

\begin{proof}
Suppose that there exists an orientation $D$ of $K_{n_1,n_2,n_3,n_4}$ whose competition graph is triangle-free.
Let $V_1,\ldots,V_4$ be the partite sets of $D$ with $|V_i|=n_i$ for each $1\leq i \leq 4$.
We take a vertex $v_i$ in $V_i$ for each $1\leq i \leq 4$.
 Suppose, to the contrary, that
  $n_2 \geq 2$. Then $n_1 \geq 2$. 
    We may take a vertex $v'_1$ (resp.\ $v'_2$) distinct from $v_1$ (resp.\ $v_2$) in $V_1$ (resp.\ $V_2$) so that the subdigraph induced by $\{v_1,v'_1,v_2,v'_2,v_3,v_4\}$ is a $4$-partite tournament $T$. 
   Then $T$ has $13$ arcs.
Since $|V(T)|=6$, at least one vertex of $T$ has indegree at least $3$, which is a contradiction.
   Thus at most one partite set of $D$ has size at least $2$.
    Hence $n_2=n_3=n_4=1$.
   Therefore $|A(D)|=3n_1+3$.
   By Lemma~\ref{lem:number-edges-triangle-free}, $|A(D)| \leq 2(|V(D)|) = 2(n_1+3)$.
   Thus $ 3n_1+3 \leq 2(n_1+3)$, so $n_1 \leq 3$.

To reach a contradiction, suppose $n_1=3$.
%% 이 내용 위에랑 겹침, 즉 traingle -free 면, k_{3,1,1,1}에서 변의 개수 5개 이하임.
   Then $D$ is an orientation of $K_{3,1,1,1}$.
   Since $|V(D)|=6$ and $|A(D)|=12$,
   \begin{equation} \label{eq:thm:charact-nonconnected-4-partite-1}
   d^-(v)=2
   \end{equation}
   for each vertex $v$ in $D$.
    We note that $|E(C(D))|\leq 6$ by Lemma~\ref{lem:number-edges-triangle-free}.
   Suppose $|E(C(D))|=6$.
   Then, since $|V(D)|=6$, each pair of vertices shares at most one common out-neighbor in $D$.
   Since $n_1=3$ and each vertex in $V_1$ has indegree $2$ by~\eqref{eq:thm:charact-nonconnected-4-partite-1},
   each vertex in $V_1$ is a common out-neighbor of $v_i$ and $v_j$ for some $i,j \in \{2,3,4\}$.
   Therefore $v_i$ and $v_j$ have a common out-neighbor in $V_1$ for each $2 \leq i \neq j \leq 4$.
   Thus $\{v_2,v_3,v_4\}$ forms a triangle in $C(D)$, which is a contradiction.
   Hence $|E(C(D))| \neq6$ and so $|E(C(D))| \leq 5$.
   Then, there exists at least one pair of vertices which has two distinct common out-neighbors by~\eqref{eq:thm:charact-nonconnected-4-partite-1}.
   Since each vertex in $V_1$ has outdegree $1$ by~\eqref{eq:thm:charact-nonconnected-4-partite-1},
   such a pair of vertices belongs to $\{v_2,v_3,v_4\}$.
   %%the -> such a 로 변경. 이런 pair가  unique 보장 없기 때문에 ...
   Without loss of generality, we may assume $\{v_2,v_3\}$ is such a pair.
   Let $x$ and $y$ be their distinct common out-neighbors of $v_2$ and $v_3$.
   Then \[N^-(x)=N^-(y)=\{v_2,v_3\}\] by~\eqref{eq:thm:charact-nonconnected-4-partite-1}.
   Since each vertex in $D$ has outdegree at least $1$ by~\eqref{eq:thm:charact-nonconnected-4-partite-1},
   $x$ and $y$ belong to the same partite set by Lemma~\ref{lem:same-out-in-neighbor-same-partite} and so $\{x,y\} \subset V_1$.
   Thus $N^+(x)=N^+(y)=\{v_4\}$.
    Hence $\{x,y\} \subseteq N^-(v_4)$ and so,    by~\eqref{eq:thm:charact-nonconnected-4-partite-1},
   $N^-(v_4)=\{x,y\}$.
   Therefore $N^+(v_4)=\{v_2,v_3,z\}$ where $z$ is a vertex in $D$ distinct from $x$ and $y$ in $V_1$.
   Without loss of generality,
   we may assume
   \[(v_3,v_2)\in A(D).\]
   Then $v_2$ is a common out-neighbor of $v_3$ and $v_4$.
Therefore by~\eqref{eq:thm:charact-nonconnected-4-partite-1}, $(v_2,z)\in A(D)$.
%    and so $(z,v_3)\in A(D)$.
Hence $z$ is a common out-neighbor of $v_2$ and $v_4$.
   Thus $\{v_2,v_3,v_4\}$ forms a triangle in $C(D)$, which is a contradiction.
   Therefore $|V_1| \neq 3$ and so $|V_1|\leq 2$.
\end{proof}

\begin{Prop} \label{prop:2-partite}
(Kim \cite{kim2016competition}). Let $D$ be an orientation of a bipartite graph with bipartition $(V_1,V_2)$.
Then the competition graph of $D$ has no edges between the vertices in $V_1$ and the vertices in $V_2$.
\end{Prop}

\begin{Thm} \label{thm:complete-triangle-free-multipartite}
Let $G$ be a connected and triangle-free graph.
Then $G$ is the competition graph of a $k$-partite tournament for some $k \geq 2$ if and only if $k \in \{3,4,5\} $ and $G$ is isomorphic to a graph belonging to the following set:
\[
 \begin{cases}
    \{ G_1,G_2,G_3,G_4,P_6,C_6 \}  & \mbox{if $k=3$;}   \\
  \{P_5, K_{1,3}, G_2 \}  & \mbox{if $k=4$;}
\\
 \{ C_5 \}  & \mbox{if $k=5$,}
  \end{cases}
 \]
 where $K_{1,3}$ is a star graph with four vertices and $G_1,G_2,G_3,$ and $G_4$ are the graphs given in Figure~\ref{fig:tri-fre}.
 \end{Thm}
 \begin{proof}
 Let $D$ be a $k$-partite tournament whose competition graph is connected and triangle-free for some $k \geq 2$.
 Then, $k \in \{3,4,5\}$ by Proposition~\ref{prop:2-partite} and Lemma~\ref{lem:connected-k-condition}.
% Take a vertex $v_i$ in $V_i$ for each $1\leq i \leq k$.
If $k = 3$, then $C(D)$ is isomorphic to a graph in $\{ G_1,G_2,G_3,G_4,P_6,C_6 \}$ by Lemma~\ref{thm:tri-fre}.
% $k=4$ and $5$ // $k=4,5$에서 정정함.
If $k=5$, then $C(D)$ is isomorphic to $C_5$ by Lemma~\ref{lem:triangle-free-5-partite-tournament}.

Suppose $k = 4$. Let $V_1, V_2,V_3, $ and $V_4$ be the partite sets of $D$. Without loss of generality, we may assume $n_1 \geq n_2 \geq n_3 \geq n_4 $ where $|V_i|=n_i$ for each $1\leq i \leq 4$.
Then  $n_1\leq 2$ and $n_2=n_3=n_4=1$ by Lemma~\ref{lem:condition-of-sizes-4-partite sets}.

   {\it Case 1.} $n_1=2$.
   Then $|V(D)|=5$ and $|A(D)|=9$.
   Therefore $|E(C(D))| \leq 4$ by Lemma~\ref{lem:number-edges-triangle-free}.
   Since $C(D)$ is connected, $|E(C(D))| \geq 4$ and so $|E(C(D))|=4$.
   Therefore $C(D)$ is a tree.
Thus $C(D)$ is isomorphic to a path graph, $G_2$, or a star graph.
Suppose, to the contrary, that $C(D)$ is a star graph.
Then there exists a center $v$ in $C(D)$.
Since $v$ has degree $4$ in $C(D)$,
$d^+(v) \geq 4$.
Then $v \in V_2 \cup V_3 \cup V_4$ and so $d^+(v)=4$ and $d^-(v)=0$.
Since $C(D)$ is triangle-free, each vertex in $D$ has indegree at most $2$.
Therefore $|A(D)| \leq 8$ and so we have reached a contradiction.
Thus $C(D)$ is isomorphic to $P_5$ or $G_5$.

   {\it Case 2.} $n_1=1$.
   Then $|A(D)|=6$ and so, by Lemma~\ref{lem:number-edges-triangle-free}, $|E(C(D))| \leq 3$.
   Since $C(D)$ is connected, $|E(C(D))| \geq 3$ and so $|E(C(D))|=3$.
   Therefore $C(D)$ is a path graph or a star graph.
   If $C(D)$ is a path graph, then
   the complement of $C(D)$ is a path graph, which contradicts Proposition~\ref{prop:no-path-on=four-tournament}.
   Therefore $C(D)$ is a star graph $K_{1,3}$.

Now we show the ``if'' part.
The competition graph of the $5$-tournament given in Figure~\ref{fig:n=5:tri-fre-1} is $C_5$.
For the $4$-partite tournaments $D_5$, $D_6$, and $D_7$ given in Figure~\ref{fig:n=4:tri-fre},
it is easy to check that $C(D_5) \cong K_{1,3}$, $C(D_6) \cong P_5$, and $C(D_7) \cong G_5$.
Each graph in $\{ G_1,G_2,G_3,G_4,P_6,C_6 \}$ is the competition graph of a tripartite tournament by Lemma~\ref{thm:tri-fre}.
Hence we have shown that the ``if" part is true.
 \end{proof}
 \begin{figure}
\begin{center}
\begin{tikzpicture}[auto,thick, scale=1.2]
    \tikzstyle{player}=[minimum size=5pt,inner sep=0pt,outer sep=0pt,draw,circle]
    \tikzstyle{source}=[minimum size=5pt,inner sep=0pt,outer sep=0pt,ball color=black, circle]
    \tikzstyle{arc}=[minimum size=5pt,inner sep=1pt,outer sep=1pt, font=\footnotesize]
    \draw (270:2cm) node (name) {$D_5$};
  \path (0:1cm)   node [player]  (v1) {};
  \path (90:1cm)   node [player]  (v2) {};
  \path (180:1cm)   node [player]  (v3) {};
  \path (270:1cm)   node [player]  (v4) {};
\draw[black,thick,-stealth] (v1) - + (v2);
\draw[black,thick,-stealth] (v2) - + (v4);
\draw[black,thick,-stealth] (v3) - + (v2);
\draw[black,thick,-stealth] (v3) - + (v4);
\draw[black,thick,-stealth] (v3) - + (v1);
\draw[black,thick,-stealth] (v4) - + (v1);
\end{tikzpicture}
\hspace{3em}
\begin{tikzpicture}[auto,thick, scale=1.2]
    \tikzstyle{player}=[minimum size=5pt,inner sep=0pt,outer sep=0pt,draw,circle]
    \tikzstyle{source}=[minimum size=5pt,inner sep=0pt,outer sep=0pt,ball color=black, circle]
    \tikzstyle{arc}=[minimum size=5pt,inner sep=1pt,outer sep=1pt, font=\footnotesize]
     \draw (270:2cm) node (name) {$D_6$};
       \path (0:1cm)   node [player]  (v1) {};
    \path (72:1cm)   node [player]  (v2) {};
    \path (144:1cm)   node [player]  (v3) {};
    \path (216:1cm)    node [player]  (v4) {};
    \path (288:1cm)   node [player]  (v5) {};

    \draw[black,thick,-stealth] (v1) - + (v3);
    \draw[black,thick,-stealth] (v2) - + (v4);
    \draw[black,thick,-stealth] (v2) - + (v5);
    \draw[black,thick,-stealth] (v3) - + (v2);
    \draw[black,thick,-stealth] (v3) - + (v4);
    \draw[black,thick,-stealth] (v4) - + (v5);
    \draw[black,thick,-stealth] (v4) - + (v1);
    \draw[black,thick,-stealth] (v5) - + (v1);
    \draw[black,thick,-stealth] (v5) - + (v3);
\end{tikzpicture}
\hspace{3em}
\begin{tikzpicture}[auto,thick, scale=1.2]
    \tikzstyle{player}=[minimum size=5pt,inner sep=0pt,outer sep=0pt,draw,circle]
    \tikzstyle{source}=[minimum size=5pt,inner sep=0pt,outer sep=0pt,ball color=black, circle]
    \tikzstyle{arc}=[minimum size=5pt,inner sep=1pt,outer sep=1pt, font=\footnotesize]
     \draw (270:2cm) node (name) {$D_7$};
       \path (0:1cm)   node [player]  (v1) {};
    \path (72:1cm)   node [player]  (v2) {};
    \path (144:1cm)   node [player]  (v3) {};
    \path (216:1cm)    node [player]  (v4) {};
    \path (288:1cm)   node [player]  (v5) {};

    \draw[black,thick,-stealth] (v1) - + (v5);
    \draw[black,thick,-stealth] (v2) - + (v4);
    \draw[black,thick,-stealth] (v5) - + (v2);
    \draw[black,thick,-stealth] (v3) - + (v2);
    \draw[black,thick,-stealth] (v3) - + (v1);
    \draw[black,thick,-stealth] (v3) - + (v4);
    \draw[black,thick,-stealth] (v4) - + (v5);
    \draw[black,thick,-stealth] (v4) - + (v1);
    \draw[black,thick,-stealth] (v5) - + (v3);
\end{tikzpicture}
\caption{Three digraphs $D_5$, $D_6$, and $D_7$ which are orientations of $K_{1,1,1,1}$,  $K_{2,1,1,1}$, and $K_{2,1,1,1}$, respectively, and whose competition graphs are isomorphic to $K_{1,3}$, $P_5$, and $G_5$, respectively}\label{fig:n=4:tri-fre}
\end{center}
\end{figure}
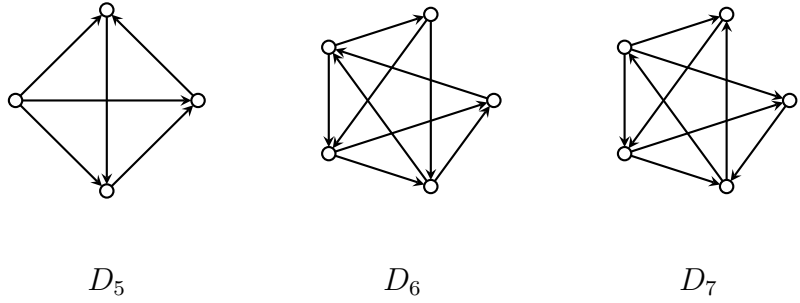

\section{The disconnected triangle-free competition graphs of multipartite tournaments}

\subsection{Bipartite tournaments}

 \begin{Lem} \label{lem:sizes-of-2partite}
    Let $n_1$ and $n_2$ be positive integers such that $n_1 \geq n_2$.
     Suppose that there exists an orientation $D$ of $K_{n_1,n_2}$ whose competition graph is  triangle-free. Then
     one of the following holds:  (a) $n_2=1$; (b) $n_2=2$;
        (c) $n_1 \leq 6$ and $n_2=3$;
         (d) $n_1 =4$ and $n_2=4$.
 \end{Lem}
        \begin{proof} It is easy to check that
         $|A(D)|=n_1 n_2$.
         Then, by Lemma~\ref{lem:number-edges-triangle-free},
   %     \begin{eqnarray*}
%        % \nonumber % Remove numbering (before each equation)
         $ n_1 n_2  \leq 2(n_1 + n_2 )$.
    %    \end{eqnarray*}
        Thus
        \begin{eqnarray}\label{eq:lem:sizes-of-2partite}
           (n_1-2)(n_2-2)\leq 4.
        \end{eqnarray}
        Then it is easy to check that $n_2 \leq 4$.
        If $n_2=1$ or $n_2=2$,
        then $n_1$ can be any positive number satisfying the inequality $n_1 \geq n_2$.
        If $n_2=3$,
        then $n_1 \leq 6$.
        If $n_2=4$,
        then $n_1=4$.
        \end{proof}

\begin{Prop} \label{prop:2-partite-path}
(Kim \cite{kim2016competition}).
Let $m$ and $n$ be positive integers such that $m \geq n$.
Then $P_m \cup P_n$ is the competition graph of a bipartite tournament if and only if $(m,n)$ is one of $(1,1)$, $(2,1)$, $(3,3)$, and $(4,3)$.
\end{Prop}

\begin{Prop} \label{prop:2-partite-cycle}
(Kim \cite{kim2016competition}).
Let $m$ and $n$ be positive integers greater than or equal to $3$.
Then $C_m \cup C_n$ is the competition graph of a bipartite tournament if and only if $(m,n)=(4,4)$.
\end{Prop}

%d^+(u) 표현이 |N^+(u)| 보다 보기에 더 편안하여, 선호함.

We give a complete characterization for a triangle-free which is a competition graph of a bipartite tournament.
We denote the set of $k$ isolated vertices in a graph by $I_k$.
%%P_1 I_k로 바꿀 수 있는 것들 다시 확인
        \begin{Thm} \label{thm:charact-nonconnected-2-partite}
Let $G$ be a triangle-free graph.
Then $G$ is the competition graph of a bipartite tournament if and only if $G$ is isomorphic to one of the followings:
\begin{itemize}
%%순서 논의 //
\item[(a)] an empty graph of order at least $2$
\item[(b)] $P_2$ with at least one isolated vertex
\item[(c)] $ P_2 \cup P_2 $ with at least one isolated vertex
\item[(d)] $ P_3 \cup P_2 $ with at least one isolated vertex
\item[(e)] $ P_2 \cup P_2 \cup P_2 $ with at least one isolated vertex
\item[(f)] $P_3 \cup  I_2$
\item[(g)]  $P_3 \cup P_3$
\item[(h)] $ P_4 \cup P_3$
\item[(i)] $P_3 \cup P_2 \cup P_2$
\item[(j)] $C_4 \cup C_4$
\item[(k)] $P_2\cup P_2 \cup P_2 \cup P_2$.
%%문장이 아니라서 세마이콜론 없어도 됨. 그래서 지웠음.
\end{itemize}

 \end{Thm}

 \begin{proof}
 We first show the ``only if" part.
Let $D$ be an orientation of $K_{n_1,n_2}$ with $n_1 \geq n_2$ whose competition graph is $G$.
%C(D) 와 다른 노테이션들이 많아서 G로 바꿈, 보기에 불편해져서...
Let $V_1=\{u_1,\ldots,u_{n_1}\}$ and $V_2=\{v_1,\ldots,v_{n_2}\}$ be the partite sets of $D$.
By Proposition~\ref{prop:2-partite},
$G$ is disconnected and there is no edge between the vertices in $V_1$ and the vertices in $V_2$.
Since $G$ is triangle-free, by Lemma~\ref{lem:sizes-of-2partite}, there are four cases to consider; $n_2=1$; $n_2=2$; $n_1 \leq 6$ and $n_2=3$; $n_1 =4$ and $n_2=4$.

{\it Case 1}. $n_2=1$.
Then, since each vertex in $D$ has indegree at most $2$, $G$ is an empty graph of order at least $2$
or $P_2$ with at least one isolated vertex.

{\it Case 2}. $n_2=2$.
Then
$G$ has at most two edges between the vertices in $V_1$.
We denote by $H$ the subgraph obtained from the subgraph $G[V_1]$ induced by $V_1$ by removing isolated vertices in it, if any.
Then $H$ is isomorphic to $P_2$ or $P_3$ or $P_2 \cup P_2$.
%Therefore the nontrivial components in $G[V_1]$ can be $P_2$ or $P_3$ or $P_2 \cup P_2$. 이 문장에서 P_2 \cup P_2가 어색하여 바꿈.
Suppose $n_1 \geq 5$.
Then, since each vertex in $V_2$ has indegree at most $2$, each vertex in $V_2$ has outdegree at least $n_1-2$. Therefore $v_1$ and $v_2$ have a common out-neighbor $u_i$ in $D$ for some $i \in \{1,\ldots, n_1\}$. Thus $v_1$ and $v_2$ are adjacent and $u_i$ is isolated in $G$.
Hence $G$ is isomorphic to $P_2 \cup I_{n-2}$ or $ P_2 \cup P_2 \cup I_{n-4}$ or $ P_3 \cup P_2 \cup I_{n-5}$ or $ P_2 \cup P_2 \cup P_2 \cup I_{n-6}$.

Now we suppose $n_1 \leq 4$.
If $H \cong P_2 \cup P_2$,
then $G[V_1] \cong P_2 \cup P_2$ and so $G \cong P_2 \cup P_2 \cup I_2$ since the two vertices in $V_2$ has no common out-neighbor.
Suppose $H \cong P_2$.
If $G[V_1]$ has two isolated vertices, then at least one of them is a common out-neighbor of $v_1$ and $v_2$ and so $G \cong P_2 \cup P_2 \cup I_2$.
If $G[V_1]$ has exactly one isolated vertex, then $G \cong P_2 \cup P_2 \cup I_1 $ or $G \cong P_2 \cup I_3$.
If $H \cong G[V_1]$, then $G \cong P_2 \cup I_2$.
Suppose $H \cong P_3$.
If $G[V_1]$ has an isolated vertex, then it must be a common out-neighbor of $v_1$ and $v_2$ and so $G \cong P_3 \cup P_2 \cup I_1$.
If $H \cong G[V_1]$, then
$G \cong P_3 \cup I_2$.

{\it Case 3}. $n_1 \leq 6$ and $n_2=3$.
Suppose, to the contrary, that $n_1 \geq 5$.
Since each vertex in $D$ has indegree at most $2$, each vertex in $V_2$ has outdegree at least $n_1-2$.
Since $n_1-2 > n_1 / 2$,
any pair of vertices in $V_2$ has a common out-neighbor in $V_1$.
Therefore the vertices in $V_2$ form a triangle, which is a contradiction.
Thus $n_1 =3$ or $n_1=4$.

{\it Subcase 3-1.} $n_1 =3$.
Then, since each vertex in $D$ has indegree at most $2$, \begin{equation} \label{eq:thm:charact-nonconnected-2-partite-1}d^+(v) \geq 1 \end{equation} for each vertex $v$ in $D$.
To reach a contradiction, we suppose that $G$ has at least three isolated vertices.
Then at least two isolated vertices belong to the same partite set in $D$.
Without loss of generality, we may assume that
$V_1$ has two isolated vertices $u_1$ and $u_2$.
Since $|V_1|=3$, $u_3$ is also isolated in $G$.
Then, since $|V_2|=3$, each vertex in $V_1$ has exactly one out-neighbor by~\eqref{eq:thm:charact-nonconnected-2-partite-1} and the out-neighbors of the vertices in $V_1$ are distinct.
Therefore any pair of the vertices in $V_2$ has a common out-neighbor in $V_1$, which implies that the vertices in $V_2$ form a triangle in $G$.
Thus $G$ has at most two isolated vertices.
Hence $G$ is isomorphic to $P_3 \cup P_3$ or $P_3 \cup P_2 \cup I_1$ or $P_2 \cup P_2 \cup I_2$.

{\it Subcase 3-2.} $n_1=4$.
Then, since each vertex in $D$ has indegree at most $2$, \begin{equation} \label{eq:thm:charact-nonconnected-2-partite-4}d^+(v) \geq 2 \end{equation} for each vertex $v$ in $V_2$.
We first suppose that there exists a vertex in $V_2$ which is isolated in $G$.
Without loss of generality, we may assume $v_1$ is an isolated vertex in $G$.
Then, since $n_1=4$,
$d^+(v_1)=2$ by~\eqref{eq:thm:charact-nonconnected-2-partite-4}.
Without loss of generality, we may assume $N^+(v_1)=\{u_1,u_2\}$.
Then, since $v_1$ is isolated in $G$, $N^+(v_2)=N^+(v_3)=\{u_3,u_4\}$ by~\eqref{eq:thm:charact-nonconnected-2-partite-4}.
Therefore $G[V_2]$ is isomorphic to $I_1 \cup P_2$ and $G[V_1]$ is isomorphic to $P_2 \cup P_2 $.
Thus $G$ is isomorphic to $I_1 \cup P_2 \cup P_2 \cup P_2$.

Now we suppose that each vertex in $V_2$ is not isolated in $G$.
Then $G[V_2]$ is isomorphic to $P_3$.
Without loss of generality, we may assume that
$G[V_2]$ is the path $v_1v_2v_3$.
Then $D$ contains a subdigraph isomorphic to $D'$ given in Figure~\ref{fig:subdigraph-4:3}.
%%그래프 모양 bipartite에 맞게 바꾸어 놓기~
\begin{figure}
\begin{center}
\begin{tikzpicture}[auto,thick, scale=1.2]
    \tikzstyle{player}=[minimum size=5pt,inner sep=0pt,outer sep=0pt,draw,circle]
     \tikzstyle{player2}=[minimum size=3pt,inner sep=0pt,outer sep=0pt,fill,color=black, circle]
    \tikzstyle{source}=[minimum size=5pt,inner sep=0pt,outer sep=0pt,ball color=black, circle]
    \tikzstyle{arc}=[minimum size=5pt,inner sep=1pt,outer sep=1pt, font=\footnotesize]
\tikzset{middlearrow/.style={decoration={
  markings,
  mark=at position 0.95 with
  {\arrow{#1}},
  },
  postaction={decorate}
  }
  }

      \draw (270:1cm) node (name) {$D'$};
    \path (1,-0.5)   node [player]  [label=right:$v_3$](u3) {};
    \path (1,0)    node [player] [label=right:$v_2$] (u2) {};
    \path (1,0.5)    node [player]  [label=right:$v_1$](u1) {};
    \path (-1,-0.5)     node [player][label=left:$u_4$]  (v4) {};
    \path (-1,-0)   node [player] [player][label=left:$u_3$] (v3) {};
  \path (-1,0.5)   node [player] [player][label=left:$u_2$] (v2) {};
  \path (-1,1)   node [player] [player][label=left:$u_1$] (v1) {};
    \draw[middlearrow={stealth}]  (u1) - + (v1);
   \draw[middlearrow={stealth}] (u2) - + (v1);
    \draw[middlearrow={stealth}] (u2) - + (v2);
    \draw[middlearrow={stealth}] (u3) - + (v2);
    \draw[middlearrow={stealth}] (v1) - + (u3);
    \draw[middlearrow={stealth}] (v2) - + (u1);
\end{tikzpicture}
\hspace{2em}
\begin{tikzpicture}[auto,thick, scale=1.2]
    \tikzstyle{player}=[minimum size=5pt,inner sep=0pt,outer sep=0pt,draw,circle]
    \tikzstyle{source}=[minimum size=5pt,inner sep=0pt,outer sep=0pt,ball color=black, circle]
    \tikzstyle{arc}=[minimum size=5pt,inner sep=1pt,outer sep=1pt, font=\footnotesize]
\tikzset{middlearrow/.style={decoration={
  markings,
  mark=at position 0.95 with
  {\arrow{#1}},
  },
  postaction={decorate}
  }
  }

    \draw (270:1cm) node (name) {$D''$};
    \path (1,-0.5)   node [player]  [label=right:$v_3$](u3) {};
    \path (1,0)    node [player] [label=right:$v_2$] (u2) {};
    \path (1,0.5)    node [player]  [label=right:$v_1$](u1) {};
    \path (-1,-0.5)     node [player][label=left:$u_4$]  (v4) {};
    \path (-1,0)   node [player] [player][label=left:$u_3$] (v3) {};
  \path (-1,0.5)   node [player] [player][label=left:$u_2$] (v2) {};
  \path (-1,1)   node [player] [player][label=left:$u_1$] (v1) {};

    \draw[middlearrow={stealth}] (u1) - + (v1);
   \draw[middlearrow={stealth}] (u2) - + (v1);
    \draw[middlearrow={stealth}] (u2) - + (v2);
    \draw[middlearrow={stealth}] (u3) - + (v2);
    \draw[middlearrow={stealth}] (v1) - + (u3);
    \draw[middlearrow={stealth}] (v2) - + (u1);

    \draw[middlearrow={stealth}] (u1) - + (v3);
    \draw[middlearrow={stealth}] (v3) - + (u3);
    \draw[middlearrow={stealth}] (v4) - + (u1);
    \draw[middlearrow={stealth}] (u3) - + (v4);
\end{tikzpicture}

\caption{ Digraphs $D'$ and $D''$ in the proof of Theorem~\ref{thm:charact-nonconnected-2-partite}.  }\label{fig:subdigraph-4:3}
\end{center}
\end{figure}
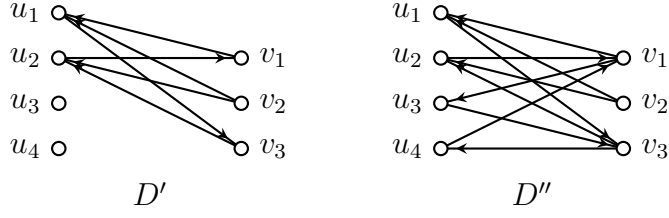
We may assume that $D'$ itself is a subdigraph of $D$.
Then, by~\eqref{eq:thm:charact-nonconnected-2-partite-4},
$N^+(v_1)\cap \{u_3,u_4\} \neq \emptyset $ and $N^+(v_3)\cap \{u_3,u_4\} \neq \emptyset $.
Since $v_1$ and $v_3$ are not adjacent in $G$, those intersections are disjoint.
%since => 결론 내리고 ASSUME 으로 써야.
We may assume that
 $N^+(v_1)\cap \{u_3,u_4\} =\{u_3\} $ and
 $N^+(v_3)\cap \{u_3,u_4\} =\{u_4\} $.
Then $D$ contains the subdigraph $D''$ given in Figure~\ref{fig:subdigraph-4:3}.
Then $v_1$ (resp.\ $v_3$) is a common out-neighbor of $u_2$ and $u_4$ (resp.\ $u_1$ and $u_3$).
If $v_2$ is a common out-neighbor of $u_3$ and $u_4$, then $G[V_1]$ is the path $u_1u_3u_4u_2$ and so $G$ is isomorphic to $P_4 \cup P_3$.
If $v_2$ is not a common out-neighbor of $u_3$ and $u_4$, then $G[V_1]$ is the union of two paths $u_1u_3$ and $u_2u_4$, and so $G$ is isomorphic to $P_3 \cup P_2 \cup P_2$.

{\it Case 4}. $n_1 =4$ and $n_2=4$.
Then $|A(D)|=16$.
Noting that $|V(D)|=8$ and each vertex has indegree at most $2$,
%간결하게, we have 남용 X
we have
\begin{equation}\label{eq:thm:charact-nonconnected-2-partite-2} d^-(v)=2 \end{equation}
for each vertex $v$ in $D$.
%좋은 영어 표현들 적어서, 보관해 놓기. u is linked by v
Then, for each vertex $v$ in $D$,
\begin{equation}\label{eq:thm:charact-nonconnected-2-partite-3}
d^+(v)=2 \end{equation}
since $v$ is adjacent to four vertices in $D$.

{\it Subcase 4-1.} $|E(G[V_1])| \geq 4$.
Then $|E(G[V_1])| = 4$ by~\eqref{eq:thm:charact-nonconnected-2-partite-2} and
$G[V_1]$ is isomorphic to $C_4$ since $G$ has no triangle.
Without loss of generality, we may assume
$G[V_1]=u_1u_2u_3u_4u_1$.
Without loss of generality, we may assume that
$N^-(v_1)=\{u_1,u_2\}$, $N^-(v_2)=\{u_2,u_3\}$, $N^-(v_3)=\{u_3,u_4\}$, and $N^-(v_4)=\{u_4,u_1\}$ by~\eqref{eq:thm:charact-nonconnected-2-partite-2}.
Therefore all arcs in $D$ are determined and so $G[V_2]$ is a $4$-cycle $v_1v_2v_3v_4v_1$.
Thus $G$ is isomorphic to $C_4\cup C_4$.

{\it Subcase 4-2.} $|E(G[V_1])|\leq 3$.
Since $|V_2|=4$, there exists a pair of vertices in $V_2$ which shares the same in-neighborhood by~\eqref{eq:thm:charact-nonconnected-2-partite-2}.
Without loss of generality, we may assume
$N^-(v_1)=N^-(v_2)=\{u_1,u_2\}$.
Then
$N^+(u_1)=N^+(u_2)=\{v_1,v_2\}$
by~\eqref{eq:thm:charact-nonconnected-2-partite-3}.
Therefore
$N^+(u_3)=N^+(u_4)=\{v_3,v_4\}$
by~\eqref{eq:thm:charact-nonconnected-2-partite-2} and~\eqref{eq:thm:charact-nonconnected-2-partite-3}.
Then $N^-(u_3)=N^-(u_4)=\{v_1,v_2\}$.
Thus it is easy to check that
$G$ is isomorphic to $P_2 \cup P_2 \cup P_2 \cup P_2$.
Hence we have shown that the ``only if" part is true.

To show the ``if" part, we fix a positive integer $k$.
Let $D_8$ be a bipartite tournament with the partite sets $\{u_1,\ldots,u_{k}\}$ and $\{v\}$, and the arc set
\[A(D_8)= \{(v,u_i) \mid 1\leq i \leq k\}\]
   (see the digraph $D_{8}$ given in Figure~\ref{fig:2-partite-tri-fre-non-connected} for an illustration). Then $C(D_8$) is an empty graph of order $k+1$.

Let $D_9$ be a bipartite tournament with the partite sets $\{u_1,\ldots,u_{k+1}\}$ and $\{v\}$, and the arc set
\[A(D_9)=\{(u_1,v),(u_{2},v)\}\cup \{(v,u_i) \mid 2 < i \leq k+1\}\]
   (see the digraph $D_{9}$ given in Figure~\ref{fig:2-partite-tri-fre-non-connected} for an illustration). Then $C(D_9$) is the path $u_1u_2$ with $k$ isolated vertices.

Let $D_{10}$ be a bipartite tournament with the partite sets $\{u_1,\ldots,u_{k+2}\}$ and $\{v_1,v_2\}$, and the arc set
\[A(D_{10})= \{(u_i,v_j) \mid  1\leq i,j\leq 2 \} \cup \{(v_i,u_j) \mid 1\leq i\leq 2, 3\leq j \leq k+2 \}\]
   (see the digraph $D_{10}$ given in Figure~\ref{fig:2-partite-tri-fre-non-connected} for an illustration). Then $C(D_{10}$) is the paths $u_1u_2$ and $v_1v_2$ with $k$ isolated vertices.

Let $D_{11}$ be a bipartite tournament with the partite sets $\{u_1,\ldots,u_{k+3}\}$ and $\{v_1,v_2\}$, and the arc set
\begin{align*}
A(D_{11})&=\{(u_1,v_1),(u_2,v_1),(u_2,v_2),(u_3,v_2),(v_1,u_3),
(v_2,u_1)\}
\\
&\ \ \ \ \ \cup \{(v_i,u_j) \mid 1\leq i \leq2, 4\leq j \leq k+3\}
\end{align*}
   (see the digraph $D_{11}$ given in Figure~\ref{fig:2-partite-tri-fre-non-connected} for an illustration). Then $C(D_{11}$) is the paths $u_1u_2u_3$ and $v_1v_2$ with $k$ isolated vertices.

   Let $D_{12}$ be a bipartite tournament with the partite sets $\{u_1,\ldots,u_{k+4}\}$ and $\{v_1,v_2\}$, and the arc set
   \begin{align*}
   A(D_{12})&= \{(u_i,v_1),(v_2,u_i) \mid i=1,2\}   \cup \{(u_i,v_2),(v_1,u_i) \mid i=3,4\}\\
&\ \ \ \ \ \cup \{(v_i,u_j) \mid 1\leq i \leq2, 5\leq j \leq k+4\}
\end{align*}
   (see the digraph $D_{12}$ given in Figure~\ref{fig:2-partite-tri-fre-non-connected} for an illustration). Then $C(D_{12}$) is the paths $u_1u_2, u_3u_4$, and $v_1v_2$ with $k$ isolated vertices.

   The competition graph of the digraph $D_{13}$ given in Figure~\ref{fig:2-partite-tri-fre-non-connected} is isomorphic to $P_3 \cup I_2$.
By Proposition~\ref{prop:2-partite-path}, there exists a bipartite tournament whose competition graph is isomorphic to $P_3 \cup P_3$.
%이제 h,i,k하면 된다.
By the way, bipartite tournaments whose competition graphs are isomorphic to $P_4 \cup P_3$ and $P_3 \cup P_2 \cup P_2$, respectively, were constructed in the subcase 3-2.
By Proposition~\ref{prop:2-partite-cycle}, there exists a bipartite tournament whose competition graph is isomorphic to $C_4 \cup C_4$.
It is easy to check that the competition graph of the bipartite tournament $D_{14}$ given in Figure~\ref{fig:2-partite-tri-fre-non-connected} is the disjoint union of the paths $u_1u_2,u_3u_4,v_1v_2$, and $v_3v_4$.
Hence we have shown that the ``if" part is true.
 \end{proof}
 %D_8~D_12까지 검은색 점들 더 모아서 표현할 것. D_9에서 점이 없을 수도 있지만, 흰색 점 하나 쓰고 u_{k+1}로 하여서 표현할 것. //D_12 특히 꼭짓점들의 위치를 조정 또는 화살표를 조정해서 잘 보이도록 하자.
\begin{figure} %[H]
\begin{center}
\begin{tikzpicture}[auto,thick, scale=1.2]
    \tikzstyle{player}=[minimum size=5pt,inner sep=0pt,outer sep=0pt,draw,circle]
     \tikzstyle{player2}=[minimum size=3pt,inner sep=0pt,outer sep=0pt,fill,color=black, circle]
    \tikzstyle{source}=[minimum size=5pt,inner sep=0pt,outer sep=0pt,ball color=black, circle]
    \tikzstyle{arc}=[minimum size=5pt,inner sep=1pt,outer sep=1pt, font=\footnotesize]

     \draw (270:1.5cm) node (name) {$D_{8}$};
       \path (0:1cm)   node [player]  (u1) [label=right:$v$] {};

    \path (130:1cm)   node [player]  (v1) [label=left:$u_1$] {};

    \path (150:1cm)   node [player2]  (v2) {};
    \path (165:1cm)   node [player2]  (v3) {};
    \path (180:1cm)   node [player2]  (v4) {};
    \path (200:1cm)   node [player]  (vk) [label=left:$u_{k}$]{};

   \draw[black,thick,-stealth] (u1) - + (v1);
   \draw[black,thick,-stealth] (u1) - + (vk);
\end{tikzpicture}
\hspace{2em}
\begin{tikzpicture}[auto,thick, scale=1.2]
    \tikzstyle{player}=[minimum size=5pt,inner sep=0pt,outer sep=0pt,draw,circle]
%    \tikzstyle{player2}=[minimum size=2pt,inner sep=0pt,outer sep=0pt,draw,circle]
     \tikzstyle{player2}=[minimum size=3pt,inner sep=0pt,outer sep=0pt,fill,color=black, circle]
    \tikzstyle{source}=[minimum size=5pt,inner sep=0pt,outer sep=0pt,ball color=black, circle]
    \tikzstyle{arc}=[minimum size=5pt,inner sep=1pt,outer sep=1pt, font=\footnotesize]

     \draw (270:1.5cm) node (name) {$D_{9}$};
       \path (0:1cm)   node [player]  (u1) [label=right:$v$] {};

    \path (120:1cm)   node [player]  (v1) [label=left:$u_1$] {};

    \path (150:1cm)   node [player]  (v2)  [label=left:$u_{2}$]{};
    \path (170:1cm)   node [player2]  (v3) {};
    \path (185:1cm)   node [player2]  (v4) {};
    \path (200:1cm)   node [player2]  (vk) {};
    \path (220:1cm)   node [player]  (vk+1)  [label=left:$u_{k+1}$] {};
   \draw[black,thick,-stealth] (v1) - + (u1);
   \draw[black,thick,-stealth] (v2) - + (u1);
   \draw[black,thick,-stealth] (u1) - + (vk+1);
\end{tikzpicture}
\hspace{2em}
\begin{tikzpicture}[auto,thick, scale=1.2]
    \tikzstyle{player}=[minimum size=5pt,inner sep=0pt,outer sep=0pt,draw,circle]

     \tikzstyle{player2}=[minimum size=3pt,inner sep=0pt,outer sep=0pt,fill,color=black, circle]
    \tikzstyle{source}=[minimum size=5pt,inner sep=0pt,outer sep=0pt,ball color=black, circle]
    \tikzstyle{arc}=[minimum size=5pt,inner sep=1pt,outer sep=1pt, font=\footnotesize]

   \tikzset{middlearrow/.style={decoration={
  markings,
  mark=at position 0.95 with
  {\arrow{#1}},
  },
  postaction={decorate}
  }
  }

     \draw (270:1.5cm) node (name) {$D_{10}$};
     %%P_2 \cup P_2 \cup isolated vertices
       \path (0:1cm)   node [player]  (u1) [label=right:$v_1$] {};
    \path (315:1cm)   node [player] [label=right:$v_2$] (u2) {};

    \path (120:1cm)   node [player]  (v1) [label=left:$u_1$] {};

    \path (150:1cm)   node [player]  (v2)  [label=left:$u_2$]{};
    \path (180:1cm)   node [player]  (v3)  [label=left:$u_3$]{};

       \path (240:1cm)   node [player]  (vk+1)  [label=left:$u_{k+2}$] {};

    \path (195:1cm)   node [player2]  (v4){};
    \path (210:1cm)   node [player2]  (v5) {};
    \path (225:1cm)   node [player2]  (v6) {};

   \draw[middlearrow={stealth}] (v1) - + (u1);
  \draw[middlearrow={stealth}] (v2) - + (u1);
  \draw[middlearrow={stealth}] (v1) - + (u2);
 \draw[middlearrow={stealth}] (v2) - + (u2);
   \draw[black,thick,-stealth] (u1) - + (v3);
   \draw[black,thick,-stealth] (u2) - + (v3);
   \draw[black,thick,-stealth] (u1) - + (vk+1);
   \draw[black,thick,-stealth] (u2) - + (vk+1);
\end{tikzpicture}

\vspace{3em}
\begin{tikzpicture}[auto,thick, scale=1.2]
    \tikzstyle{player}=[minimum size=5pt,inner sep=0pt,outer sep=0pt,draw,circle]
%    \tikzstyle{player2}=[minimum size=2pt,inner sep=0pt,outer sep=0pt,draw,circle]
     \tikzstyle{player2}=[minimum size=3pt,inner sep=0pt,outer sep=0pt,fill,color=black, circle]
    \tikzstyle{source}=[minimum size=5pt,inner sep=0pt,outer sep=0pt,ball color=black, circle]
    \tikzstyle{arc}=[minimum size=5pt,inner sep=1pt,outer sep=1pt, font=\footnotesize]

   \tikzset{middlearrow/.style={decoration={
  markings,
  mark=at position 0.95 with
  {\arrow{#1}},
  },
  postaction={decorate}
  }
  }

     \draw (270:1.5cm) node (name) {$D_{11}$};
     %%P_3 \cup P_2 \cup isolated vertices
    \path (0:1cm)   node [player] (u1)[label=right:$v_1$] {};
    \path (315:1cm)  node [player] (u2)[label=right:$v_2$] {};

    \path (90:1cm)   node [player]  (v1) [label=above:$u_1$] {};
    \path (120:1cm)   node [player]  (v2)  [label=left:$u_2$]{};
    \path (150:1cm)   node [player]  (v3)  [label=left:$u_3$]{};
    \path (180:1cm)   node [player]  (v4)  [label=left:$u_4$]{};
    \path (195:1cm)   node [player2]  (v5) {};
    \path (210:1cm)   node [player2]  (v6) {};
    \path (225:1cm)   node [player2]  (v7) {};
    \path (240:1cm)   node [player]  (vk+1) [label=left:$u_{k+3}$] {};
   \draw[middlearrow={stealth}] (v1) - + (u1);
   \draw[middlearrow={stealth}] (v2) - + (u1);
   \draw[black,thick,-stealth] (u1) - + (v3);

\draw[black,thick,-stealth] (u2) - + (v1);
\draw[middlearrow={stealth}] (v2) - + (u2);
\draw[middlearrow={stealth}] (v3) - + (u2);
\draw[black,thick,-stealth] (u1) - + (vk+1);
\draw[black,thick,-stealth] (u2) - + (vk+1);
\draw[black,thick,-stealth] (u1) - + (v4);
\draw[black,thick,-stealth] (u2) - + (v4);
\end{tikzpicture}
\hspace{2em}
\begin{tikzpicture}[auto,thick, scale=1.2]
    \tikzstyle{player}=[minimum size=5pt,inner sep=0pt,outer sep=0pt,draw,circle]
%    \tikzstyle{player2}=[minimum size=2pt,inner sep=0pt,outer sep=0pt,draw,circle]
     \tikzstyle{player2}=[minimum size=3pt,inner sep=0pt,outer sep=0pt,fill,color=black, circle]
    \tikzstyle{source}=[minimum size=5pt,inner sep=0pt,outer sep=0pt,ball color=black, circle]
    \tikzstyle{arc}=[minimum size=5pt,inner sep=1pt,outer sep=1pt, font=\footnotesize]
     \draw (270:1.5cm) node (name) {$D_{12}$};
          %%P_2 \cup P_2 \cup P_2 \cup isolated vertices
     \tikzset{middlearrow/.style={decoration={
  markings,
  mark=at position 0.95 with
  {\arrow{#1}},
  },
  postaction={decorate}
  }
  }
       \path (5:1cm)   node [player]  (u1) [label=right:$v_1$] {};
    \path (330:1cm)   node [player] [label=right:$v_2$] (u2) {};

    \path (90:1cm)   node [player]  (v1) [label=above:$u_1$] {};

    \path (120:1cm)   node [player]  (v2)  [label=above:$u_2$]{};
    \path (150:1cm)   node [player]  (v3)  [label=left:$u_3$]{};
    \path (180:1cm)   node [player]  [label=left:$u_4$]  (v4){};
    \path (210:1cm)   node [player]
    [label=left:$u_5$] (v5) {};
    \path (225:1cm)   node [player2]  (v6) {};
    \path (240:1cm)   node [player2]  (v7) {};
    \path (255:1cm)   node [player2]  (v8) {};
    \path (270:1cm)   node [player]
    [label=right:$u_{k+4}$] (vk+1) {};
   \draw[middlearrow={stealth},thick] (v1) - + (u1);
   \draw[middlearrow={stealth}] (v2) - + (u1);
   \draw[middlearrow={stealth}] (v3) - + (u2);
   \draw[middlearrow={stealth}] (v4) - + (u2);
   \draw[black,thick,-stealth] (u1) - + (v5);
   \draw[black,thick,-stealth] (u2) - + (v5);
 \draw[black,thick,-stealth] (u1) - + (vk+1);
   \draw[black,thick,-stealth] (u2) - + (vk+1);
\draw[black,thick,-stealth] (u2) - + (v1);
\draw[black,thick,-stealth] (u2) - + (v2);
\draw[black,thick,-stealth] (u1) - + (v3);
\draw[black,thick,-stealth] (u1) - + (v4);
\end{tikzpicture}
\hspace{2em}
\begin{tikzpicture}[auto,thick, scale=1.2]
    \tikzstyle{player}=[minimum size=5pt,inner sep=0pt,outer sep=0pt,draw,circle]
%    \tikzstyle{player2}=[minimum size=2pt,inner sep=0pt,outer sep=0pt,draw,circle]
     \tikzstyle{player2}=[minimum size=3pt,inner sep=0pt,outer sep=0pt,fill,color=black, circle]
    \tikzstyle{source}=[minimum size=5pt,inner sep=0pt,outer sep=0pt,ball color=black, circle]
    \tikzstyle{arc}=[minimum size=5pt,inner sep=1pt,outer sep=1pt, font=\footnotesize]

     \tikzset{middlearrow/.style={decoration={
  markings,
  mark=at position 0.95 with
  {\arrow{#1}},
  },
  postaction={decorate}
  }
  }
     \draw (270:1.5cm) node (name) {$D_{13}$};
       \path (0:1cm)   node [player]  (u1) [label=right:$v_1$] {};

    \path (150:1cm)   node [player]  (v1) [label=left:$u_1$] {};

    \path (180:1cm)   node [player] (v2) [label=left:$u_2$]  {};
    \path (210:1cm)   node [player] (v3) [label=left:$u_3$] {};
    \path (330:1cm)   node [player] (u2) [label=right:$v_2$]   {};

\draw[middlearrow={stealth}]  (v1) - + (u1);
\draw[middlearrow={stealth}]  (v2) - + (u1);
\draw[middlearrow={stealth}]  (v2) - + (u2);
\draw[middlearrow={stealth}]  (v3) - + (u2);
\draw[black,thick,-stealth] (u1) - + (v3);
\draw[black,thick,-stealth] (u2) - + (v1);
\end{tikzpicture}

\vspace{3em}
\begin{tikzpicture}[auto,thick, scale=1.5]
    \tikzstyle{player}=[minimum size=5pt,inner sep=0pt,outer sep=0pt,draw,circle]
     \tikzstyle{player2}=[minimum size=3pt,inner sep=0pt,outer sep=0pt,fill,color=black, circle]
    \tikzstyle{source}=[minimum size=5pt,inner sep=0pt,outer sep=0pt,ball color=black, circle]
    \tikzstyle{arc}=[minimum size=5pt,inner sep=1pt,outer sep=1pt, font=\footnotesize]
\tikzset{middlearrow/.style={decoration={
  markings,
  mark=at position 0.95 with
  {\arrow{#1}},
  },
  postaction={decorate}
  }
  }
  %https://ipfs-sec.stackexchange.cloudflare-ipfs.com/tex/A/question/39278.html
     \draw (0,-0.5) node (name) {$D_{14}$};
    \path (1,1.5)   node [player] (u1)[label=right:$v_1$] {};
    \path (1,1)  node [player] (u2)[label=right:$v_2$] {};
    \path (1,0.5)   node [player]  [label=right:$v_3$](u3) {};
    \path (1,0)   node [player]      [label=right:$v_4$](u4) {};

    \path (-1,1.5)   node [player]  (v1) [label=left:$u_1$] {};
    \path (-1,1)   node [player]  (v2)  [label=left:$u_2$]{};
    \path (-1,0.5)   node [player]  (v3)  [label=left:$u_3$]{};
    \path (-1,0)   node [player]      [label=left:$u_4$](v4) {};

   \draw[middlearrow={stealth}] (v1) - + (u1);
   \draw[middlearrow={stealth}] (v2) - + (u1);
   \draw[middlearrow={stealth}] (v1) - + (u2);
   \draw[middlearrow={stealth}] (v2) - + (u2);

    \draw[middlearrow={stealth}] (v3) - + (u3);
    %{stealth reversed} 화살표방향 거꾸로됨
   \draw[middlearrow={stealth}] (v3) - + (u4);
   \draw[middlearrow={stealth}](v4) - + (u3);
   \draw[middlearrow={stealth}] (v4) - + (u4);

       \draw[middlearrow={stealth}] (u1) - + (v3);
   \draw[middlearrow={stealth}] (u1) - + (v4);
   \draw[middlearrow={stealth}] (u2) - + (v3);
   \draw[middlearrow={stealth}] (u2) - + (v4);

      \draw[middlearrow={stealth}] (u3) - + (v1);
   \draw[middlearrow={stealth}] (u4) - + (v1);
   \draw[middlearrow={stealth}] (u3) - + (v2);
   \draw[middlearrow={stealth}] (u4) - + (v2);
\end{tikzpicture}
\caption{ Bipartite tournaments in the proof of Theorem~\ref{thm:charact-nonconnected-2-partite} }\label{fig:2-partite-tri-fre-non-connected}
\end{center}
\end{figure}
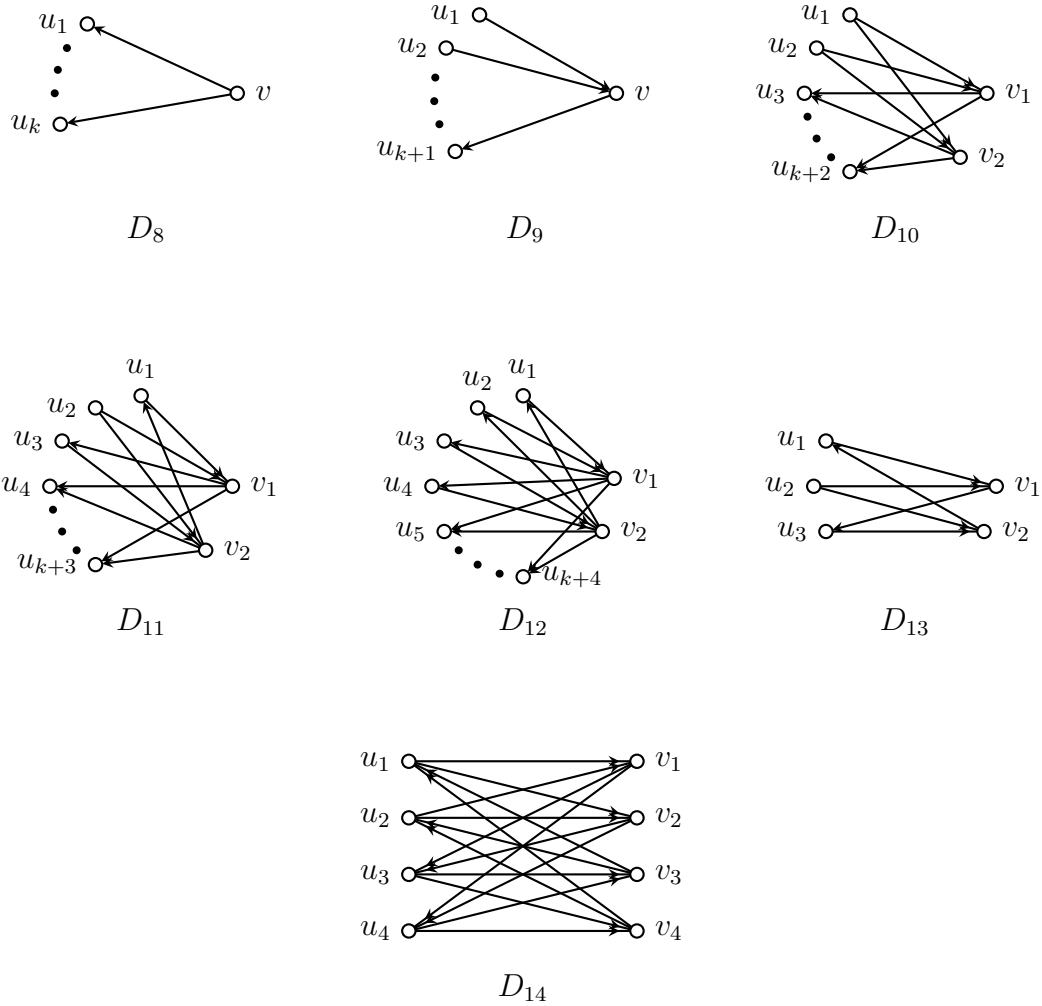
%%Then 2번 나올 때, Then now로 고쳤음.
\subsection{$k$-partite tournaments for $k \geq 3$}
By Lemmas~\ref{lem:connected-k-condition} and \ref{lem:triangle-free-5-partite-tournament},
the following lemma is true.

\begin{Lem} \label{lem:disconect-condition}
If the competition graph of a $k$-partite tournament is triangle-free and disconnected for some positive integer $k \geq 3$,
then $k=3$ or $k=4$.
\end{Lem}

By Lemma~\ref{lem:disconect-condition},
it is sufficient to consider tripartite tournaments and $4$-partite tournaments for studying disconnected triangle-free competition graphs of multipartite tournaments.

\begin{Lem} \label{lem:condition-outdegree-traingle-free}
Let $D$ be a multipartite tournament whose competition graph is triangle-free.
Suppose that a vertex $v$ is contained in a partite set $X$ of $D$. Then $|V(D)|- |X| -2 \leq d^+(v)$.
\end{Lem}

\begin{proof}
Since $C(D)$ is triangle-free,
$d^-(v) \leq 2$.
Then, since $D$ is a multipartite tournament, $d^-(v) = |V(D)|-|X|-d^+(v)$ and so
$|V(D)|- |X| -2 \leq d^+(v)$.
\end{proof}

The following is immediately true by Lemma~\ref{lem:condition-outdegree-traingle-free}.
%%cor로 바꾸기!
\begin{Cor} \label{cor:out-degree-1-4-partite}
If the competition graph of a $4$-partite tournament $D$ is triangle-free, then each vertex has outdegree at least $1$ in $D$.
\end{Cor}

\begin{Lem} \label{lem:condition-of-indegree-1}
Let $D$ be a multipartite tournament whose competition graph is triangle-free.
If $m$ is the number of vertices of indegree $1$ in $D$, then $2|V(D)| - |A(D)|  \geq m $.
\end{Lem}

\begin{proof}
Let $m$ be the number of vertices of indegree $1$ in $D$.
Since $C(D)$ is triangle-free,
each vertex has indegree at most $2$.
Therefore
\[|A(D)|=\sum_{v\in V(D)} d^-(v) \leq 2(|V(D)|-m)+m  =2|V(D)|-m.  \]
\end{proof}

Now we are ready to introduce one of our main theorems.
\begin{Thm} \label{thm:charact-nonconnected-4-partite}
Let $G$ be a disconnected and triangle-free graph.
Then $G$ is the competition graph of a $4$-partite tournament if and only if $G$ is isomorphic to $P_3 \cup P_2$ or  $ P_3 \cup I_1$.
%P_i가 일반적으로 쓰이는 기호이므로 굳이 i에 조건을 써서 나타낼 필요가 없음. is a path -> denotes 로 변경.
\end{Thm}

\begin{proof}
We first show the ``only if" part.
Let $D$ be an orientation of $K_{n_1,n_2,n_3,n_4}$ with $n_1 \geq \cdots \geq n_4$ whose competition graph $C(D)$ is disconnected and triangle-free.
%녀such n_1 \geq n_4 에서 with로 바꿈.
Let $V_1,\ldots,V_4$ be the partite sets of $D$ with $|V_i|=n_i$ for each $1\leq i \leq 4$.
Then $n_1 \leq 2$ and $n_2=n_3=n_4=1$ by Lemma~\ref{lem:condition-of-sizes-4-partite sets}.

{\it Case 1}. $n_1=2$.
  Then $|A(D)|=9$.
   Therefore $|E(C(D))| \leq 4$ by Lemma~\ref{lem:number-edges-triangle-free}.
   Let $l$ and $m$ be the number of isolated vertices in $C(D)$ and the number of vertices of indegree $1$ in $D$, respectively.
   By Corollary~\ref{cor:out-degree-1-4-partite},
   each vertex has outdegree at least $1$, so each isolated vertex in $C(D)$ has an out-neighbor
  in $D$.
  Yet, since each out-neighbor of an isolated vertex has indegree $1$,
   $l \leq m$.
   By Lemma~\ref{lem:condition-of-indegree-1},
   $m \leq 1$.
   Therefore $l \leq 1.$

   Suppose, to the contrary, that $l=1$.
    Then $m=1$.
   Let $w$ be the isolated vertex in $C(D)$.
   Since each vertex in $N^+(w)$ has indegree $1$, $d^+(w) \leq 1$.
   Since each vertex has outdegree at least $1$, $d^+(w)=1$.
   Since $C(D)$ is triangle-free, $d^-(w) \leq 2$ and so $w  \in V_1$.
   Let $V_1=\{v_1,w\}$, $V_2=\{v_2\}$, $V_3=\{v_3\}$ and $V_4=\{v_4\}$.
   Without loss of generality,
   we may assume $N^+(w)=\{v_2\}$.
   Then \[N^-(w)=\{v_3,v_4\}.\]
   Since $w$ is an isolated vertex in $C(D)$,
    $N^-(v_2)=\{w\}$ and so $N^+(v_2)=\{v_1,v_3,v_4\}$.
   Without loss of generality,
   we may assume $(v_3,v_4)\in A(D)$.
   Then, since $d^-(v_4) \leq 2$, \[N^-(v_4)=\{v_2,v_3\},\] and so
   $(v_4,v_1)\in A(D)$.
   Therefore \[N^-(v_1)=\{v_2,v_4\}.\] Thus $\{v_2,v_3,v_4\}$ forms a triangle in  $C(D)$, which is a contradiction.
    Hence $l=0$.
   Since $C(D)$ is disconnected and $|V(D)|=5$,
   $C(D)$ has two components each of which has $2$ and $3$ vertices, respectively.
   Then, one of the components must be $P_2$.
   On the other hand, since $C(D)$ is triangle-free, the other component is isomorphic to $P_3$.
  Therefore
   $C(D)$ is isomorphic to $P_3 \cup P_2$.

   {\it Case 2}. $n_1=1$.
   Then $D$ is an orientation of $K_{1,1,1,1}$, which is a tournament.
   By Proposition~\ref{prop:minimum-number-of-edges},
   $|E(C(D))| \geq 2$.
 By the way, since $|A(D)|= 6$,
   $|E(C(D))| \leq 3$ by Lemma~\ref{lem:number-edges-triangle-free}.
   Therefore $|E(C(D))|=2$ or $3$.
   % 2 \leq E(C(D))\leq 3 에서 바꿈, 어차피 2개인데, 굳이 부등식을 쓸 필요는 없음
   Thus $C(D)$ has exactly two components and so is isomorphic to $I_1 \cup P_3$ or $P_2 \cup P_2$.
   If $C(D)$ is isomorphic to $P_2 \cup P_2$,
   then $D$ has two vertices $a$ and $b$ such that $d^-(a)=d^-(b)=2$ and $N^-(a) \cap N^-(b) = \emptyset$, which is impossible for a digraph of order four.
   Therefore $C(D)$ is isomorphic to $I_1 \cup P_3$.

   To show the ``if" part, we consider the $4$-partite tournaments $D_{15}$ and $D_{16}$ given in Figure~\ref{fig:n=4:tri-fre-non-connected}.
   It is easy to check that $C(D_{15}) \cong P_2 \cup P_3$, and $C(D_{16}) \cong I_1 \cup P_3$.
Therefore we have shown that the ``if" part is true.
\end{proof}

\begin{figure}
\begin{center}
\begin{tikzpicture}[auto,thick, scale=1.2]
    \tikzstyle{player}=[minimum size=5pt,inner sep=0pt,outer sep=0pt,draw,circle]
    \tikzstyle{source}=[minimum size=5pt,inner sep=0pt,outer sep=0pt,ball color=black, circle]
    \tikzstyle{arc}=[minimum size=5pt,inner sep=1pt,outer sep=1pt, font=\footnotesize]

    \draw (270:1.5cm) node (name) {$D_{15}$};
    \path (0:1cm)   node [player]  (x1) {};
    \path (72:1cm)   node [player]  (y1) {};
    \path (144:1cm)   node [player]  (y2) {};
    \path (216:1cm)    node [player]  (z1) {};
    \path (288:1cm)   node [player]  (z2) {};

    \draw[black,thick,-stealth] (x1) - + (z1);
    \draw[black,thick,-stealth] (x1) - + (z2);
    \draw[black,thick,-stealth] (x1) - + (y1);

    \draw[black,thick,-stealth] (y1) - + (z1);
    \draw[black,thick,-stealth] (y1) - + (z2);

    \draw[black,thick,-stealth] (y2) - + (x1);
    \draw[black,thick,-stealth] (y2) - + (y1);

    \draw[black,thick,-stealth] (z1) - + (y2);
    \draw[black,thick,-stealth] (z2) - + (y2);
\end{tikzpicture}
\hspace{5em}
\begin{tikzpicture}[auto,thick, scale=1.2]
    \tikzstyle{player}=[minimum size=5pt,inner sep=0pt,outer sep=0pt,draw,circle]
    \tikzstyle{source}=[minimum size=5pt,inner sep=0pt,outer sep=0pt,ball color=black, circle]
    \tikzstyle{arc}=[minimum size=5pt,inner sep=1pt,outer sep=1pt, font=\footnotesize]

    \draw (270:1.5cm) node (name) {$D_{16}$};
    \path (0:1cm)   node [player]  (x) {};
    \path (90:1cm)   node [player]  (y) {};
    \path (180:1cm)   node [player]  (z) {};
    \path (270:1cm)    node [player]  (w) {};

    \draw[black,thick,-stealth] (x) - + (z);
    \draw[black,thick,-stealth] (x) - + (w);
    \draw[black,thick,-stealth] (z) - + (w);
    \draw[black,thick,-stealth] (w) - + (y);
    \draw[black,thick,-stealth] (y) - + (x);
    \draw[black,thick,-stealth] (y) - + (z);
\end{tikzpicture}
\caption{ The digraphs $D_{15}$ and $D_{16}$ in the proof of Theorem~\ref{thm:charact-nonconnected-4-partite}  }\label{fig:n=4:tri-fre-non-connected}
\end{center}
\end{figure}
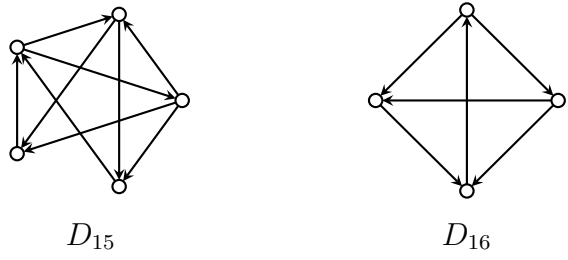
By Lemma~\ref{lem:disconect-condition},
it only remains to characterize disconnected and triangle-free competition graphs of tripartite tournaments.
The following theorem lists all the disconnected and triangle-free competition graphs of tripartite tournaments
\begin{Thm} \label{thm:charact-nonconnected-3-partite}
%P_1 => isolated vertices로
Let $G$ be a disconnected and triangle-free graph.
Then $G$ is the competition graph of a tripartite tournament if and only if $G$ is isomorphic to one of the followings:
\begin{itemize}
\item[(a)] an empty graph of order $3$
\item[(b)]
$P_2$ with at least one isolated vertex
\item[(c)]  $P_3$ with at least one isolated vertex
\item[(d)]   $P_4$ with at least one isolated vertex
\item[(e)]  $K_{1,3} \cup I_1$
\item[(f)]  $K_{1,3} \cup P_2$
\item[(g)]  $P_2 \cup P_4$
\item[(h)]  $P_2 \cup P_2$ with at least one isolated vertex
\item[(i)]  $P_2 \cup P_3$ with or without isolated vertices
\item[(j)]  $P_2 \cup P_2 \cup P_2$.
%%문장이 아니라서 세마이콜론 없어도 됨. 그래서 지웠음.
\end{itemize}
\end{Thm}

\begin{proof}
To show the ``only if" part, suppose that $D$ is an orientation of $K_{n_1,n_2,n_3}$ whose competition graph is disconnected and triangle-free where $n_1$, $n_2$, and $n_3$ are positive integers such that $n_1 \geq n_2 \geq n_3$.
 Then, by Lemma~\ref{lem:sizes-of-3partite},
 $(n_1,n_2,n_3) \in A \cup  \{(2,2,1), (2,2,2), (3,2,1)\} $ where $A=\{(m,1,1)\mid \text{$m$ is a positive integer}\}$.
 %Lemma의 순서대로 무조건 따라가는 거 아니고, 증명의 흐름이 자연스러운 거대로 가는 거 / 난이도도 고려하지만/

 {\it Case 1}. $(n_1,n_2,n_3) \in A$.
 Let  $V_1:=\{x_1,\ldots,x_{n_1}\}$, $V_2:=\{y\}$, and $V_3:=\{z\}$ be the partite sets of $D$.
 Without loss of generality, we may assume
 \[(y,z)\in A(D).\]
  Suppose that $C(D)$ is an empty graph. Then $y$ is an isolated vertex in $C(D)$. If $d^+(y) \geq 2$, then a vertex in $V_1$ should be an out-neighbor of $y$ and so $y$ is adjacent to one of the vertex and $z$ in $C(D)$, which is a contradiction. Therefore $d^+(y)=1$ and so $N^+(y)=\{z\}$.
If $n_1 \geq 2$, then the in-neighbors of $y$ are adjacent in $C(D)$, which is a contradiction.
Therefore $n_1=1$. Thus $C(D)$ is isomorphic to three isolated vertices.

Now we suppose that $C(D)$ is not an empty graph.
 Then $y$ is the only possible neighbor of $z$ in $C(D)$. %% z의 입장
Moreover, $z$ and at most one vertex in $V_1$ are the only possible neighbors of $y$ in $C(D)$. %% y의 입장
Since $z$ has indegree at most $2$ in $D$,
$y$ is the only possible common out-neighbor of two vertices in $V_1$.
%it is possible that => possibly로 고침.
Therefore only one pair of vertices in $V_1$ is possibly adjacent in $C(D)$.
Hence the following are the only possible graphs isomorphic to $C(D)$:
%위아래 모두 결론을 내려서, hence로 함 / thus보다는
\begin{itemize}
\item an empty graph of order $3$
\item $P_4$ with at least one isolated vertex
\item $P_3 \cup P_2$ with at least one isolated vertex
\item $P_3$ with at least one isolated vertex
\item $P_2\cup P_2$ with at least one isolated vertex
\item $P_2$ with at least one isolated vertex.
\end{itemize}
% 기존에 길었던 증명들이 간단하게 줄여짐. Case가 깔끔하지 않게 나누어졌었는데, Case를 이루는 경우들에 대해 point를 생각해서, 표현할 수 있었음. 그리고 empty graph 경우에는 특별하게 더 살펴봤었음.

{\it Case 2}. $(n_1,n_2,n_3)=(2,2,1)$.
Suppose, to the contrary, that $C(D)$ has $l$ isolated vertices for some positive integer $l \geq 2$.
By Lemma~\ref{lem:condition-outdegree-traingle-free},
each vertex in $D$ has outdegree at least $1$.
By the way, since $2|V(D)|-|A(D)|=2$,
$D$ has at most two vertices of indegree $1$ by Lemma~\ref{lem:condition-of-indegree-1}.
Then, since each out-neighbor of isolated vertices
has indegree $1$,
$l \leq 2$ and so $l=2$.
Let $u_1$ and $u_2$ be the isolated vertices in $C(D)$.
Then, $d^+(u_1) =d^+(u_2) =1$.
Suppose that $u_1$ and $u_2$ are contained in distinct partite sets in $D$.
Without loss of generality,
we may assume $(u_1,u_2)\in A(D)$.
Then $d^-(u_2)=1$.
However, since $d^+(u_2)=1$, $d^+(u_2)+d^-(u_2)=2 \neq |V(D) \setminus X|$  where $X$ is a partite set containing $u_2$, which is impossible.
Therefore $u_1$ and $u_2$ belong to the same partite set of $D$. Then $\{u_1,u_2\} \subseteq V_1$ or $\{u_1,u_2\} \subseteq V_2$. Without loss of generality, we may assume $\{u_1,u_2\} \subseteq V_1$.
Then $V_1=\{u_1,u_2\}$.
Let $v_1$ and $v_2$ be the out-neighbors of $u_1$ and $u_2$, respectively.
Then $N^-(v_1)=\{u_1\}$ and $N^-(v_2)=\{u_2\}$.
Therefore there is no arc between $v_1$ and $v_2$, and so $v_1$ and $v_2$ belong to the same partite set $V_2$.
Then $V_2=\{v_1,v_2\}$.
Since $d^-(v_1)=d^-(v_2)=1$, the vertex $z$, in the remaining partite set of $D$, is a common out-neighbor of $v_1$ and $v_2$.
Then $N^-(z)=\{v_1,v_2\}$
and so $N^+(z)=\{u_1,u_2\}$.
Therefore $u_1$ (resp.\ $u_2$) is a common out-neighbor of $v_2$ (resp.\ $v_1$) and $z$.
Thus $\{v_1,v_2,z\}$ forms a triangle in $C(D)$, which is a contradiction.
%%여기서 D가 다 결정되어서 / K_1 \cup K_1 \cup K_3가 됨. (2,i) digraph를 탐구한다면 이게 포함 됨.
Hence $C(D)$ has at most one isolated vertex.
Therefore $C(D)$ has at most three components.

Since $|A(D)|=8$, $|E(C(D))| \leq 4$ by Lemma~\ref{lem:number-edges-triangle-free}.
If $|E(C(D))| \leq 1$, then $C(D)$ has at least $2$ isolated vertices, which is a contradiction.
Therefore $2 \leq |E(C(D))| \leq 4$.
If $|E(C(D))|=2$, then, $C(D)$ is isomorphic to $P_2 \cup P_2\cup I_1$.
If $|E(C(D))|=3$, then $C(D)$ is isomorphic to $ P_4 \cup I_1$ or $ K_{1,3}  \cup I_1$ or $P_3 \cup P_2$.
Suppose $|E(C(D))|=4$.
 Since $|A(D)|=8$ and $|E(C(D))|=4$, there exists a vertex $w$ of indegree $0$ in $D$ and each vertex in $V(D) \setminus \{w\}$ has indegree $2$.
Moreover, for distinct vertices $a$ and $b$ of indegree $2$,
$N^-(a) \neq N^-(b) $.
Then, since $w$ has outdegree at least $3$, $w$ has degree at least $3$ in $C(D)$.
  Since $C(D)$ is disconnected and triangle-free, $K_{1,3}\cup I_1$ is the only possible graph isomorphic to $C(D)$, which contradicts the assumption that $|E(C(D))| =4$.
  Therefore $C(D)$ is isomorphic to one of
$P_2 \cup P_2 \cup I_1$ or
$ P_4 \cup I_1$ or $ K_{1,3}  \cup I_1$ or $P_3 \cup P_2$.

{\it Case 3}. $(n_1,n_2,n_3)=(2,2,2)$.
Then $|A(D)|=12$.
Since $|V(D)|=6$ and each vertex of $D$ has indegree at most $2$,
\begin{equation}
\label{eq:thm:charact-nonconnected-3-partite-4}
d^-(v)=2
\end{equation}
 and so
 \begin{equation}
\label{eq:thm:charact-nonconnected-3-partite-5}
d^+(v)=2
\end{equation}
for each vertex $v$ in $D$.
Therefore $C(D)$ has no isolated vertex.
Thus each component of $C(D)$ contains at least two vertices.
Let $t$ be the number of the components of $C(D)$.
Then $t \leq 3$.
Since $C(D)$ is disconnected, $t=2$ or $t=3$.
Suppose, to the contrary, that $t=2$.
Then, since $C(D)$ is triangle-free,
it is easy to check that $|E(C(D))| \leq 5$.
Since $|V(D)|=6$,
there exist at least two vertices $a_1$ and $a_2$ sharing the same in-neighborhood by~\eqref{eq:thm:charact-nonconnected-3-partite-4}.
Then $a_1$ and $a_2$ are contained in the same partite set and form a component in $C(D)$ by Lemma~\ref{lem:same-out-in-neighbor-same-partite}.
Then the other component must contain four vertices.
Without loss of generality, we may assume
$V_1:=\{a_1,a_2\}$ is a partite set of $D$.
Let $\{b_1,b_2\}=N^-(a_1)=N^-(a_2)$ for some vertices $b_1$ and $b_2$ in $D$. Then $N^+(b_1)=N^+(b_2)=\{a_1,a_2\}$ by~\eqref{eq:thm:charact-nonconnected-3-partite-5}.
Therefore $b_1$ and $b_2$ are contained in the same partite sets and $\{b_1,b_2\}$ forms a component in $C(D)$ by Lemma~\ref{lem:same-out-in-neighbor-same-partite},
which is a contradiction.
Therefore $t\neq 2$ and so $t=3$.
Then, since each component of $C(D)$ contains at least two vertices, $C(D)$ must be isomorphic to $P_2 \cup P_2 \cup P_2$.

{\it Case 4}. $(n_1,n_2,n_3)=(3,2,1)$.
Then $|A(D)|=11$.
Since each vertex has indegree at most $2$ in $D$, one vertex has indegree $1$ and the other vertices have indegree $2$.
%moreover, furthermore 는 parallel 한데, 더 나아가는 느낌인 반면, in addition은 완전히 parallel 임. 둘이 조금 다름!
% and $|E(C(D))| \leq 5$.
Let $V_1$, $V_2$, and $V_3$ be the partite sets of $D$ satisfying $|V_1|=3$, $|V_2|=2$, and $|V_3|=1$ and $v^*$ be the vertex of indegree $1$ in $D$.
Then, since $(n_1,n_2,n_3)=(3,2,1)$,
each vertex in $V_1$ has outdegree at least $1$ and
each vertex in $V_2 \cup V_3$ has outdegree at least $2$.

Suppose, to the contrary, that $C(D)$ has an isolated vertex $u$.
Since $v^*$ is the only vertex of indegree $1$, $N^+(u)=\{v^*\}$.
Therefore $u\in V_1$.
Then $v^* \in V_2\cup V_3$.
Suppose $v^* \in V_3$.
Since $d^-(v^*)=1$,
$N^+(v^*)=(V_1 \cup V_2) \setminus \{u\}$.
Moreover, since $V_2 \subset N^+(v^*)$ and each vertex in $D$ other than $v^*$ has indegree $2$,
each vertex in $V_2$ has an out-neighbor in $V_1\setminus \{u\}$.
Therefore each vertex in $V_2$ is adjacent to $v^*$ in $C(D)$.
By the way, since $N^+(u)=\{v^*\}$,
 $u$ is a common out-neighbor of the two vertices in $V_2$.
Therefore $V_2 \cup \{v^*\}$ forms a triangle in $C(D)$, which is a contradiction.
Thus $v^* \in V_2$.
Let $V_1=\{u,x_1,x_2\}$, $V_2=\{v^*,y\}$, and $V_3=\{z\}$.
Then $N^+(v^*)=\{x_1,x_2,z\}$ and
\[N^-(u)=\{y,z\},\] so $y$ and $z$ are adjacent in $C(D)$.
Since $d^-(y)=2$, $d^+(y)=2$.
Thus
\[N^+(v^*) \cap N^+(y) \neq \emptyset\]
and so $v^*$ and $y$ are adjacent in $C(D)$.
Since $d^-(z)=2$ and $v^* \in N^-(z)$,
$\{x_1,x_2 \} \not \subset N^-(z)$.
Therefore $N^+(z) \cap \{x_1,x_2\} \neq \emptyset$ and so \[N^+(v^*) \cap N^+(z) \neq \emptyset.\]
Thus $v^*$ and $z$ are adjacent in $C(D)$.
Hence $\{v^*,y,z\}$ forms a triangle in $C(D)$, which is a contradiction.
Consequently, we have shown that $C(D)$ has no isolated vertex, so each component in $C(D)$ has size at least two.
Then, since $|V(D)|=6$, $C(D)$ has two or three components.
If $C(D)$ has three components,
then $C(D)$ must be isomorphic to $P_2 \cup P_2 \cup P_2$.

Now we suppose that $C(D)$ has two components.
Then it is easy to check that $4\leq |E(C(D))| \leq 5$ since $C(D)$ is triangle-free and has no isolated vertices.
Suppose, to the contrary, that $|E(C(D))|=5$.
Then, since the vertices in $C(D)$ except the five vertices of indegree $2$ have indegree less than $2$, no pair of adjacent vertices in $C(D)$ have two distinct common out-neighbors in $D$.
Moreover, $C(D)$ must be isomorphic to $P_2 \cup C_4$ where $C_4$ is a cycle of length $4$.
Since only one vertex has indegree $1$ and the other vertices have indegree $2$ in $D$,
there exist two vertices $a$ and $b$ in $V_1$ which have outdegree $1$ in $D$.
Then $a$ and $b$ have at most degree $1$ in $C(D)$, so $\{a,b\}$ is a path component in $C(D)$.
Therefore
$a$ and $b$ have a common out-neighbor $c$, in $D$.
Then $a$ and $b$ are common out-neighbors of the two vertices in $V_2\cup V_3 \setminus \{c\}$, which is a contradiction.
Hence $|E(C(D))| \neq 5$ and so $|E(C(D))| = 4$.
Since five vertices have indegree $2$ in $D$,
there exists a pair of adjacent vertices which have two common out-neighbors.
Therefore
 there exist two vertices whose in-neighbors are the same. Then, since each vertex has outdegree at least $1$ by Lemma~\ref{lem:condition-outdegree-traingle-free}, the two vertices form a component in $C(D)$ by Lemma~\ref{lem:same-out-in-neighbor-same-partite}.
Thus $C(D)$ must be isomorphic to  $P_2 \cup K_{1,3}$ or $P_2 \cup P_4$.
Hence we have shown that the ``only if" part is true.

   Now we show the ``if" part.
   The competition graph of the digraph $D_{17}$ given in Figure~\ref{fig:3-partite-tri-fre-non-connected}
   is an empty graph of order $3$.

Now we fix a positive integer $k$.
Let $D_{18}$ be a tripartite tournament with the partite sets $\{w_1,\ldots,w_k\}$, $\{x\}$, $\{y\}$ and the arc set
\[A(D_{18})=\{(x,y)\}\cup \{(x,w_i),(y,w_i) \mid 1\leq i \leq k\}\]
   (see the digraph $D_{18}$ given in Figure~\ref{fig:3-partite-tri-fre-non-connected} for an illustration). Then $C(D_{18})$ is isomorphic to $P_2$ with $k$ isolated vertices.

Let $D_{19}$ be a tripartite tournament with the partite sets $\{v,w_1,\ldots,w_k\}$, $\{x\}$, $\{y\}$ and the arc set
\[A(D_{19})=\{(v,x),(v,y),(x,y)\}\cup \{(x,w_i),(y,w_i) \mid 1\leq i \leq k\}\]
   (see the digraph  $D_{19}$ given in Figure~\ref{fig:3-partite-tri-fre-non-connected} for an illustration). Then $C(D_{19})$ is the path $vxy$ together with $k$ isolated vertices.

Let $D_{20}$ be a tripartite tournament with the partite sets $\{v_1,v_2,w_1,\ldots,w_k\}$, $\{x\}$, $\{y\}$ and the arc set
\[A(D_{20})=\{(v_1,x),(v_2,x),(v_2,y),(x,y),(y,v_1)\}\cup \{(x,w_i),(y,w_i) \mid 1\leq i \leq k\}\]
   (see the digraph $D_{20}$ given in Figure~\ref{fig:3-partite-tri-fre-non-connected} for an illustration).
 Then $C(D_{20})$ is the path $v_1v_2xy$ with $k$ isolated vertices.

 The competition graphs of the digraphs $D_{21},D_{22}$ and $D_{23}$ given in Figure~\ref{fig:3-partite-tri-fre-non-connected} are isomorphic to $K_{1,3} \cup I_1$, $K_{1,3}\cup P_2$, and $P_2 \cup P_4$, respectively.

 Let $D_{24}$ be a tripartite tournament with the partite sets $\{v_1,v_2,w_1,\ldots,w_k\}$, $\{x\}$, $\{y\}$ and the arc set
\[A(D_{24})=\{(v_1,x),(v_2,x),(x,y),(y,v_1),(y,v_2)\}\cup \{(x,w_i),(y,w_i) \mid 1\leq i \leq k\}\]
   (see the digraph  $D_{24}$ given in Figure~\ref{fig:3-partite-tri-fre-non-connected} for an illustration).
 Then $C(D_{24}$) is isomorphic to $P_2 \cup P_2$ with $k$ isolated vertices.

 Let $D_{25}$ be a tripartite tournament with the partite sets $\{v_1,v_2,v_3,w_1,\ldots,w_k\}$, $\{x\}$, $\{y\}$ and the arc set
\[A(D_{25})=\{(v_1,x),(v_2,y),(v_3,x),(x,v_2),(x,y),(y,v_1),(y,v_3)\}\cup \{(x,w_i),(y,w_i) \mid 1\leq i \leq k\}\]
   (see the digraph  $D_{25}$ given in Figure~\ref{fig:3-partite-tri-fre-non-connected} for an illustration).
 Then $C(D_{25}$) is isomorphic to $P_2 \cup P_3$ with $k$ isolated vertices.

 The competition graphs of the digraphs $D_{26}$ and $D_{27}$ given in Figure~\ref{fig:3-partite-tri-fre-non-connected} are isomorphic to $P_2 \cup P_3$ and $P_2 \cup P_2 \cup P_2$, respectively.
 Hence we have shown that the ``if'' part is true.
\end{proof}

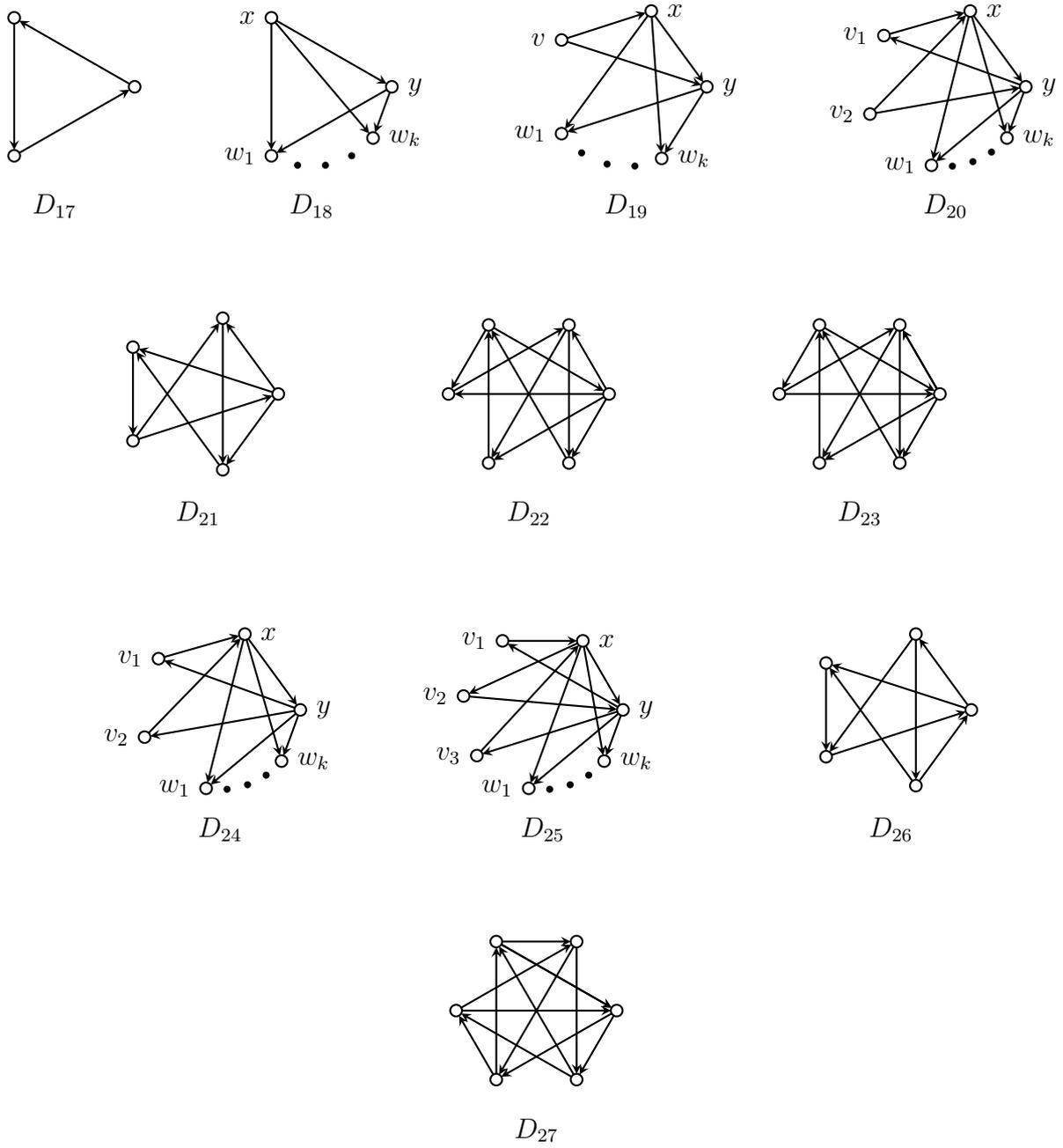
\begin{figure}
\begin{center}
\begin{tikzpicture}[auto,thick, scale=1.2]
    \tikzstyle{player}=[minimum size=5pt,inner sep=0pt,outer sep=0pt,draw,circle]
    \tikzstyle{player2}=[minimum size=2pt,inner sep=0pt,outer sep=0pt,draw,circle]
    \tikzstyle{source}=[minimum size=5pt,inner sep=0pt,outer sep=0pt,ball color=black, circle]
    \tikzstyle{arc}=[minimum size=5pt,inner sep=1pt,outer sep=1pt, font=\footnotesize]

     \draw (270:1.5cm) node (name) {$D_{17}$};
       \path (0:1cm)   node [player]  (x1) {};
    \path (120:1cm)   node [player]  (y1) {};
    \path (240:1cm)   node [player]  (z1) {};

   \draw[black,thick,-stealth] (x1) - + (y1);
    \draw[black,thick,-stealth] (y1) - + (z1);
    \draw[black,thick,-stealth] (z1) - + (x1);
\end{tikzpicture}
\hspace{2em}
%(b) P_2 \cup  isolated vertices
\begin{tikzpicture}[auto,thick, scale=1.2]
    \tikzstyle{player}=[minimum size=5pt,inner sep=0pt,outer sep=0pt,draw,circle]
     \tikzstyle{player2}=[minimum size=3pt,inner sep=0pt,outer sep=0pt,fill,color=black, circle]
    \tikzstyle{source}=[minimum size=5pt,inner sep=0pt,outer sep=0pt,ball color=black, circle]
    \tikzstyle{arc}=[minimum size=5pt,inner sep=1pt,outer sep=1pt, font=\footnotesize]

     \draw (270:1.5cm) node (name) {$D_{18}$};
       \path (0:1cm)   node [player]  (x1) [label=right:$y$] {};
    \path (120:1cm)   node [player]  (y1) [label=left:$x$] {};
    \path (240:1cm)   node [player]  (z1) [label=left:$w_1$] {};

    \path (260:1cm)   node [player2]  (z4) {};
    \path (280:1cm)   node [player2]  (z5) {};
    \path (300:1cm)   node [player2]  (z6) {};
    \path (320:1cm)   node [player]  (z7) [label=right:$w_k$]{};

   \draw[black,thick,-stealth] (x1) - + (z1);
   \draw[black,thick,-stealth] (y1) - + (z1);
   \draw[black,thick,-stealth] (y1) - + (x1);
   \draw[black,thick,-stealth] (x1) - + (z7);
   \draw[black,thick,-stealth] (y1) - + (z7);
\end{tikzpicture}
\hspace{2em}
%%(c) P_3 \cup Isolated vertices.
\begin{tikzpicture}[auto,thick, scale=1.2]
    \tikzstyle{player}=[minimum size=5pt,inner sep=0pt,outer sep=0pt,draw,circle]
     \tikzstyle{player2}=[minimum size=3pt,inner sep=0pt,outer sep=0pt,fill,color=black, circle]
    \tikzstyle{source}=[minimum size=5pt,inner sep=0pt,outer sep=0pt,ball color=black, circle]
    \tikzstyle{arc}=[minimum size=5pt,inner sep=1pt,outer sep=1pt, font=\footnotesize]

     \draw (270:1.5cm) node (name) {$D_{19}$};
       \path (0:1cm)   node [player] [label=right:$y$] (x1) {};
    \path (72:1cm)   node [player] [label=right:$x$] (y1) {};
    \path (216:1cm)   node [player]  (z1) [label=left:$w_1$]{};
    \path (144:1cm)    node [player]  (z2) [label=left:$v$] {};
%    \path (288:1cm)   node [player]  (z3) {};
    \path (236:1cm)   node [player2]  (z4) {};
    \path (256:1cm)   node [player2]  (z5) {};
    \path (276:1cm)   node [player2]  (z6) {};
 \path (296:1cm)    node [player]  (z7) [label=right:$w_k$] {};

   \draw[black,thick,-stealth] (x1) - + (z1);
    \draw[black,thick,-stealth] (z2) - + (x1);
    \draw[black,thick,-stealth] (y1) - + (x1);
    \draw[black,thick,-stealth] (y1) - + (z1);
    \draw[black,thick,-stealth] (z2) - + (y1);
       \draw[black,thick,-stealth] (x1) - + (z7);
   \draw[black,thick,-stealth] (y1) - + (z7);
\end{tikzpicture}
\hspace{2em}
%%(d) P_4 with isolated  veritces
\begin{tikzpicture}[auto,thick, scale=1.2]
    \tikzstyle{player}=[minimum size=5pt,inner sep=0pt,outer sep=0pt,draw,circle]
    \tikzstyle{player2}=[minimum size=3pt,inner sep=0pt,outer sep=0pt,fill,color=black, circle]
    \tikzstyle{source}=[minimum size=5pt,inner sep=0pt,outer sep=0pt,ball color=black, circle]
    \tikzstyle{arc}=[minimum size=5pt,inner sep=1pt,outer sep=1pt, font=\footnotesize]

     \draw (270:1.5cm) node (name) {$D_{20}$};
       \path (0:1cm)   node [player]  [label=right:$y$] (x1) {};
    \path (72:1cm)   node [player] [label=right:$x$] (y1) {};
    \path (260:1cm)   node [player]  (z1) [label=left:$w_1$]{};
    \path (200:1cm)    node [player]  (z2) [label=left:$v_2$] {};
    \path (140:1cm)   node [player]   [label=left:$v_1$]   (z3) {};
    \path (275:1cm)   node [player2]  (z4) {};
    \path (290:1cm)   node [player2]  (z5) {};
    \path (305:1cm)   node [player2]  (z6) {};
   \path (320:1cm)   node [player]   [label=right:$w_k$]   (z7) {};

    \draw[black,thick,-stealth] (x1) - + (z1);
    \draw[black,thick,-stealth] (z2) - + (x1);
    \draw[black,thick,-stealth] (y1) - + (x1);

    \draw[black,thick,-stealth] (z2) - + (y1);
    \draw[black,thick,-stealth] (y1) - + (z1);
    \draw[black,thick,-stealth] (z3) - + (y1);
    \draw[black,thick,-stealth] (x1) - + (z3);
           \draw[black,thick,-stealth] (x1) - + (z7);
   \draw[black,thick,-stealth] (y1) - + (z7);
\end{tikzpicture}

\vspace{3em}
%%%K_1,3 cup P_1
\begin{tikzpicture}[auto,thick, scale=1.2]
    \tikzstyle{player}=[minimum size=5pt,inner sep=0pt,outer sep=0pt,draw,circle]
    \tikzstyle{source}=[minimum size=5pt,inner sep=0pt,outer sep=0pt,ball color=black, circle]
    \tikzstyle{arc}=[minimum size=5pt,inner sep=1pt,outer sep=1pt, font=\footnotesize]

    \draw (270:1.5cm) node (name) {$D_{21}$};
    \path (0:1cm)   node [player]  (x1) {};
    \path (72:1cm)   node [player]  (y1) {};
    \path (144:1cm)   node [player]  (y2) {};
    \path (216:1cm)    node [player]  (z1) {};
    \path (288:1cm)   node [player]  (z2) {};

    \draw[black,thick,-stealth] (x1) - + (y1);
    \draw[black,thick,-stealth] (x1) - + (y2);
    \draw[black,thick,-stealth] (z1) - + (x1);
    \draw[black,thick,-stealth] (x1) - + (z2);

    \draw[black,thick,-stealth] (y1) - + (z2);
    \draw[black,thick,-stealth] (z1) - + (y1);

    \draw[black,thick,-stealth] (y2) - + (z1);
    \draw[black,thick,-stealth] (z2) - + (y2);
\end{tikzpicture}
\hspace{5em}
% K_{3,2,1} \cong  K_{1,3} \cup P_2
\begin{tikzpicture}[auto,thick, scale=1.2]
    \tikzstyle{player}=[minimum size=5pt,inner sep=0pt,outer sep=0pt,draw,circle]
    \tikzstyle{source}=[minimum size=5pt,inner sep=0pt,outer sep=0pt,ball color=black, circle]
    \tikzstyle{arc}=[minimum size=5pt,inner sep=1pt,outer sep=1pt, font=\footnotesize]

    \draw (270:1.5cm) node (name) {$D_{22}$};
    \path (0:1cm)   node [player]  (x1) {};
    \path (60:1cm)   node [player]  (y1) {};
    \path (120:1cm)   node [player]  (y2) {};
    \path (180:1cm)    node [player]  (z1) {};
    \path (240:1cm)   node [player]  (z2) {};
    \path (300:1cm)   node [player]  (z3) {};

    \draw[black,thick,-stealth] (x1) - + (y1);
    \draw[black,thick,-stealth] (x1) - + (z2);
    \draw[black,thick,-stealth] (x1) - + (z3);
    \draw[black,thick,-stealth] (y2) - + (x1);
    \draw[black,thick,-stealth] (x1) - + (z1);

    \draw[black,thick,-stealth] (y1) - + (z3);
    \draw[black,thick,-stealth] (y1) - + (z2);
    \draw[black,thick,-stealth] (z1) - + (y1);

    \draw[black,thick,-stealth] (y2) - + (z1);
    \draw[black,thick,-stealth] (z2) - + (y2);
    \draw[black,thick,-stealth] (z3) - + (y2);
\end{tikzpicture}
\hspace{5em}
%%P_2, P_4
\begin{tikzpicture}[auto,thick, scale=1.2]
    \tikzstyle{player}=[minimum size=5pt,inner sep=0pt,outer sep=0pt,draw,circle]
    \tikzstyle{source}=[minimum size=5pt,inner sep=0pt,outer sep=0pt,ball color=black, circle]
    \tikzstyle{arc}=[minimum size=5pt,inner sep=1pt,outer sep=1pt, font=\footnotesize]

    \draw (270:1.5cm) node (name) {$D_{23}$};
    \path (0:1cm)   node [player]  (x1) {};
    \path (60:1cm)   node [player]  (y1) {};
    \path (120:1cm)   node [player]  (y2) {};
    \path (180:1cm)    node [player]  (z1) {};
    \path (240:1cm)   node [player]  (z2) {};
    \path (300:1cm)   node [player]  (z3) {};

    \draw[black,thick,-stealth] (x1) - + (y1);
    \draw[black,thick,-stealth] (x1) - + (z2);
    \draw[black,thick,-stealth] (x1) - + (z3);
    \draw[black,thick,-stealth] (y2) - + (x1);
    \draw[black,thick,-stealth] (z1) - + (x1);

    \draw[black,thick,-stealth] (y1) - + (z3);
    \draw[black,thick,-stealth] (y1) - + (z2);
    \draw[black,thick,-stealth] (x1) - + (y1);
    \draw[black,thick,-stealth] (z1) - + (y1);

    \draw[black,thick,-stealth] (y2) - + (z1);
    \draw[black,thick,-stealth] (z2) - + (y2);
    \draw[black,thick,-stealth] (z3) - + (y2);
\end{tikzpicture}

\vspace{3em}
%% P_2 \cup P_2 \cup \bigcup P_1$
\begin{tikzpicture}[auto,thick, scale=1.2]
    \tikzstyle{player}=[minimum size=5pt,inner sep=0pt,outer sep=0pt,draw,circle]
       \tikzstyle{player2}=[minimum size=3pt,inner sep=0pt,outer sep=0pt,fill,color=black, circle]
    \tikzstyle{source}=[minimum size=5pt,inner sep=0pt,outer sep=0pt,ball color=black, circle]
    \tikzstyle{arc}=[minimum size=5pt,inner sep=1pt,outer sep=1pt, font=\footnotesize]

     \draw (270:1.5cm) node (name) {$D_{24}$};
       \path (0:1cm)   node [player]  (x1)  [label=right:$y$] {};
    \path (72:1cm)   node [player]  (y1)  [label=right:$x$] {};
    \path (140:1cm)   node [player]  (z1)  [label=left:$v_1$]{};
    \path (200:1cm)    node [player]  (z2) [label=left:$v_2$] {};
    \path (260:1cm)   node [player]  (z3)  [label=left:$w_1$] {};

      \path (275:1cm)   node [player2]  (z4) {};
    \path (290:1cm)   node [player2]  (z5) {};
    \path (305:1cm)   node [player2]  (z6) {};
   \path (320:1cm)   node [player]   [label=right:$w_k$]   (z7) {};

   \draw[black,thick,-stealth] (x1) - + (z1);
    \draw[black,thick,-stealth] (x1) - + (z2);
    \draw[black,thick,-stealth] (y1) - + (x1);
    \draw[black,thick,-stealth] (z1) - + (y1);
    \draw[black,thick,-stealth] (z2) - + (y1);

    \draw[black,thick,-stealth] (x1) - + (z3);
    \draw[black,thick,-stealth] (y1) - + (z3);

    \draw[black,thick,-stealth] (x1) - + (z7);
    \draw[black,thick,-stealth] (y1) - + (z7);
\end{tikzpicture}
\hspace{2em}
%P_3 P_2 isolated vertices
\begin{tikzpicture}[auto,thick, scale=1.2]
    \tikzstyle{player}=[minimum size=5pt,inner sep=0pt,outer sep=0pt,draw,circle]
      \tikzstyle{player2}=[minimum size=3pt,inner sep=0pt,outer sep=0pt,fill,color=black, circle]
    \tikzstyle{source}=[minimum size=5pt,inner sep=0pt,outer sep=0pt,ball color=black, circle]
    \tikzstyle{arc}=[minimum size=5pt,inner sep=1pt,outer sep=1pt, font=\footnotesize]

     \draw (270:1.5cm) node (name) {$D_{25}$};
       \path (0:1cm)   node [player]  (x1) [label=right:$y$]  {};
    \path (60:1cm)   node [player]  (y1) [label=right:$x$]  {};
    \path (260:1cm)   node [player]  (z1) [label=left:$w_1$] {};
    \path (170:1cm)    node [player]  (z2) [label=left:$v_2$]{};
    \path (215:1cm)   node [player]  (z3) [label=left:$v_3$] {};
    \path (120:1cm)   node [player]  (z4)[label=left:$v_1$]  {};
    \path (275:1cm)   node [player2]  (z5) {};
    \path (290:1cm)   node [player2]  (z6) {};
    \path (305:1cm)   node [player2]  (z7) {};
 \path (320:1cm)   node [player]  (z8) [label=right:$w_k$] {};

    \draw[black,thick,-stealth] (x1) - + (z1);
    \draw[black,thick,-stealth] (z2) - + (x1);
    \draw[black,thick,-stealth] (y1) - + (x1);

    \draw[black,thick,-stealth] (y1) - + (z2);
    \draw[black,thick,-stealth] (y1) - + (z1);
    \draw[black,thick,-stealth] (z3) - + (y1);
    \draw[black,thick,-stealth] (z4) - + (y1);
    \draw[black,thick,-stealth] (x1) - + (z4);
    \draw[black,thick,-stealth] (x1) - + (z3);

    \draw[black,thick,-stealth] (x1) - + (z8);
    \draw[black,thick,-stealth] (y1) - + (z8);
\end{tikzpicture}
\hspace{5em}
%%K_{2,2,1}\cong P_3 \cup P_2
\begin{tikzpicture}[auto,thick, scale=1.2]
    \tikzstyle{player}=[minimum size=5pt,inner sep=0pt,outer sep=0pt,draw,circle]
    \tikzstyle{source}=[minimum size=5pt,inner sep=0pt,outer sep=0pt,ball color=black, circle]
    \tikzstyle{arc}=[minimum size=5pt,inner sep=1pt,outer sep=1pt, font=\footnotesize]

    \draw (270:1.5cm) node (name) {$D_{26}$};
    \path (0:1cm)   node [player]  (x1) {};
    \path (72:1cm)   node [player]  (y1) {};
    \path (144:1cm)   node [player]  (y2) {};
    \path (216:1cm)    node [player]  (z1) {};
    \path (288:1cm)   node [player]  (z2) {};

    \draw[black,thick,-stealth] (x1) - + (y1);
    \draw[black,thick,-stealth] (x1) - + (y2);
    \draw[black,thick,-stealth] (z1) - + (x1);
    \draw[black,thick,-stealth] (z2) - + (x1);

    \draw[black,thick,-stealth] (y1) - + (z1);
    \draw[black,thick,-stealth] (y1) - + (z2);

    \draw[black,thick,-stealth] (y2) - + (z1);
    \draw[black,thick,-stealth] (z2) - + (y2);
\end{tikzpicture}

\vspace{3em}
%,K_{2,2,2} \cong P_2 \cup P_2 \cup P_2
\begin{tikzpicture}[auto,thick, scale=1.2]
    \tikzstyle{player}=[minimum size=5pt,inner sep=0pt,outer sep=0pt,draw,circle]
    \tikzstyle{source}=[minimum size=5pt,inner sep=0pt,outer sep=0pt,ball color=black, circle]
    \tikzstyle{arc}=[minimum size=5pt,inner sep=1pt,outer sep=1pt, font=\footnotesize]

    \draw (270:1.5cm) node (name) {$D_{27}$};
    \path (0:1cm)   node [player]  (x1) {};
    \path (60:1cm)   node [player]  (x2) {};
    \path (120:1cm)   node [player]  (y1) {};
    \path (180:1cm)    node [player]  (y2) {};
    \path (240:1cm)   node [player]  (z1) {};
    \path (300:1cm)   node [player]  (z2) {};

    \draw[black,thick,-stealth] (x1) - + (z1);
    \draw[black,thick,-stealth] (x1) - + (z2);
    \draw[black,thick,-stealth] (y1) - + (x1);
    \draw[black,thick,-stealth] (y2) - + (x1);  \draw[black,thick,-stealth] (x2) - + (z1);
    \draw[black,thick,-stealth] (x2) - + (z2);
    \draw[black,thick,-stealth] (y1) - + (x2);
    \draw[black,thick,-stealth] (y1) - + (x1);
    \draw[black,thick,-stealth] (y2) - + (x2);
    \draw[black,thick,-stealth] (z1) - + (y1);
    \draw[black,thick,-stealth] (z2) - + (y1);
    \draw[black,thick,-stealth] (z1) - + (y2);
    \draw[black,thick,-stealth] (z2) - + (y2);
\end{tikzpicture}
\caption{The digraphs in the proof of Theorem~\ref{thm:charact-nonconnected-3-partite}  }\label{fig:3-partite-tri-fre-non-connected}
%%너무 많아서 digraphs를 목록을 안 쓰고 digraph로 함.
\end{center}
\end{figure}

\section{Closing remarks
} \label{sec:number-of}

In this paper, we completely identified the triangle-free graphs which are the competition graphs of $k$-partite tournaments for $k \geq 2$,
following up the previous paper~\cite{choi2022competitively} in which all the complete graphs which are the competition graphs of $k$-partite tournaments for $k \geq 2 $ are found.

Taking into account the fact that a cycle is a $2$-regular graph and a complete graph of order $n$ is an $(n-1)$-regular graph,
characterizing the cubic graphs which are
the competition graphs of $k$-partite tournament for $k \geq 2 $ seems to be an interesting
research problem to be resolved in the next step.
\bibliographystyle{abbrv}
%\bibliography{competition}
%

\end{document}